\documentclass[11pt,oneside]{amsart}

\usepackage[all]{xypic}
\usepackage{amssymb}
\usepackage{verbatim}
\usepackage[utf8]{inputenc}
\usepackage[pdftex]{graphicx}

\newtheorem{thm}{Theorem}[section]
\newtheorem{cor}[thm]{Corollary}
\newtheorem{lem}[thm]{Lemma}
\newtheorem{prop}[thm]{Proposition}
\theoremstyle{remark}
\newtheorem{example}[thm]{\bf Example}
\newtheorem{remark}[thm]{\bf Remark}
\newtheorem{propdef}[thm]{\bf Proposition-Definition}

\def\N{\mathbb{N}}
\def\Z{\mathbb{Z}}

\def\ZZ{\Z \oplus \Z}
\def\R{\mathbb{R}}
\def\C{\mathbb{C}}
\def\a{\alpha}
\def\b{\beta}

\def\e{\varepsilon}
\def\f{\phi}
\def\vf{\varphi}

\def\m{\mu}

\def\w{\omega}
\def\G{\Gamma}

\def\ol{\overline}

\def\O{\Omega}
\def\proj{\pi}

\newcommand{\up}[2]{\null^{#2}\hspace{-.1em}#1}
\newcommand{\ZnZ}[1]{\mathbb{Z}/#1\mathbb{Z}}
\newcounter{rmq}[section]
\title{GROUP EXTENSIONS WITH INFINITE CONJUGACY CLASSES}

\begin{document}
\maketitle
\begin{center}
{\sc Jean-Philippe PR\' EAUX}\footnote{Laboratoire d'Analyse, Topologie et Probabilit\'es, UMR CNRS 7353,
 39 rue F.Joliot-Curie, F-13453 Marseille
cedex 13, France\\
\indent {\it E-mail :} \ preaux@cmi.univ-mrs.fr\\
{\it Mathematical subject classification : 20E45 (Primary), 20E22 (Secondary)}}
\end{center}

\markboth{Jean-Philippe Pr\'eaux}{Group Extensions with Infinite Conjugacy Classes}


\begin{abstract}
We characterize the group property of being with infinite conjugacy
classes (or {\it icc}, {\it i.e.} infinite and of which all
conjugacy classes except $\{1\}$ are  infinite)  for groups which are extensions of groups. 
We prove a general result for extensions of groups, then deduce characterizations in semi-direct products, wreath products, finite extensions, among others examples we also deduce a characterization for amalgamated products and HNN extensions.
 The icc property is correlated to the Theory of von Neumann algebras since a necessary and sufficient condition for the von Neumann algebra of a discrete group $\Gamma$ to be a factor of type $II_1$, is that $\Gamma$ be icc. 
Our approach applies in full generality to the study of icc property since any group that does not split as an extension is simple,
and in such case icc property becomes equivalent to being infinite.
\end{abstract}

\section*{\bf Introduction}

A group $\Gamma$ is said to be with {\sl infinite conjugacy classes} (or {\sl icc} for short) if $\G$ is infinite and all of its conjugacy classes, except $\{1\}$, are infinite. The icc property has been studied in several classes of discrete groups: free groups and free products in \cite{roiv}, groups acting on a Bass-Serre tree in \cite{ydc}, fundamental groups of 3-manifolds and $PD(3)$-groups in \cite{aogf3v}.

The main motivation for studying this property in discrete groups arises from the theory of von Neumann algebras (cf. \cite{dixmier}). For a discrete group  $\G$ one defines the von Neumann algebra $W_\lambda^*(\G)$ of $\G$ which yields important examples of more general {\sl von Neumann algebras}, {\it i.e.}  of $*$-algebras of bounded operators on a (separable) Hilbert space closed in the weak topology and containing the identity. The decomposition theorem of von Neumann asserts that each von Neumann algebra on a separable Hilbert space is a 'direct integral' of {\sl factors}, {\it i.e.}  of von Neumann algebras whose centers are reduced to the scalar operators $\C$. Therefore, the problem of classifying isomorphic classes of von Neumann algebras reduces to that of classifying isomorphic classes of factors. The factors fall into several types among which {\sl factors of type $II_1$} play an important role and are intensively studied. Concerning von Neumann algebras of groups, factors can only be of  type $II_1$ and are characterized by the\smallskip\\
{\bf Murray-von Neumann characterization.} {\sl A von Neumann algebra $W_\lambda^*(\G)$ of a discrete group $\G$  is a factor of type $II_1$ if and only if $\G$ is icc.}\smallskip\\
Thus icc discrete groups provide examples of factors of type $II_1$; uncountably many as stated in \cite{mcduff}. That makes  interesting asking whether an arbitrary discrete group is icc other not.

This work is concerned with the icc property for {\sl extensions of groups} (cf. \cite{rotman}). We give a characterisation of the icc property by mean of the invariants of the extension, then we particularize this general result in cases of extensions that split: semi-direct products and wreath products; among several other cases and examples. 

In section 1 we state two concise conditions that are sufficient for a group that decomposes as an extension to be icc. The former one is also necessary and we characterize it by a precise statement; we give an example proving that the latter one is not a necessary condition, despite what one should expect. In section 2 we weaken that last condition in order to obtain necessary and sufficient conditions; this yields the main result; we also give noteworthy reformulations under additional hypotheses. 
The sections 3 and 4 are devoted to particularizations of the general result to, respectively, semi-direct products and  wreath products
(both complete and restricted). In the last section we give examples of what the result becomes in particular cases, among that groups with a proper finite index subgroup, amalgams of groups and HNN extensions (recovering briefly the results stated in \cite{ydc}). 
Those last examples are an illustration that our approach applies  in full generality to the study of the icc property in groups: on the one hand a simple group is icc whenever it is infinite (cf. Proposition \ref{simplecci}), on the other hand any non-simple group decomposes
non-trivially as an extension and our results apply.

\setcounter{section}{-1}

\section{\bf Notations and preliminaries}
\label{section1}

Let $G$ be a group,  $H$ a non-empty subset of $G$, {\it e.g.} a subgroup,  and $x,y\in G$, we use notations $\null^yx:=yxy^{-1}$ and $\null^Hx:=\{\null^yx;y\in H\}$; in particular, $\null^Gx$ denotes the conjugacy class of $x$ in $G$. The centralizer of $x$ in $H$ and the center of $G$ are denoted respectively by $C_H(x)$ and $Z(G)$.
Note that the cardinal $|^Gx|$ of $\null^Gx$ equals the index $[G:C_G(x)]$ of $C_G(x)$ in $G$; indeed, the group $G$ acts transitively on $\null^Gx$, and $x$ has stabilizer $C_G(x)$. In particular, icc groups have trivial centers, and a direct product of
 groups $\not=\{1\}$  is icc if and only if each of its factors is icc.\smallskip

\begin{prop}\label{simplecci} A simple group is icc if and only if it is infinite.
\end{prop}
\begin{proof}
An icc group is infinite, so let $G$ be an infinite simple group, we need to prove that $G$ is icc. Suppose on the contrary that $G$ is not icc, and let $g\in G\setminus\{1\}$ with centralizer $C_G(g)$ of finite index in $G$. The intersection $C$ of all conjugates of $C_G(g)$ in $G$ is normal with finite index in $G$. Since $G$ is infinite $C\not= \{1\}$, since $G$ is simple $C=G$. But $C$ has a non trivial center which is impossible for an infinite simple group.
\end{proof}

\section{\bf Sufficient conditions that are not necessary}

We say that a subgroup $N$ of a group $G$ is {\sl finitely normalized} by $G$ if $N$ is normal in $G$ and the action of $G$ by conjugacy on $N$ has only finite orbits. For $K$ a normal subgroup of $G$ we denote by:
$$
FC_G(K)=\left\lbrace u\in K\,;\, |^Gu|<\infty\right\rbrace
$$ 
the union of finite $G$-conjugacy classes in $K$. It's easily seen that $FC_G(K)$ is a normal subgroup of $G$, the greatest subgroup of $K$ which is finitely normalized by $G$. Let denote by $FC(G)=FC_G(G)$ the union of finite conjugacy classes in $G$; $FC(G)$ is a characteristic subgroup of $G$ and $G\not=\{1\}$ is icc if and only if $FC(G)=\{1\}$.\medskip\\ 
\noindent
Obviously $FC_G(K)\not=\{1\}$ implies that $G$ is not icc. The condition $FC_G(K)\not=\{1\}$ is characterized by:
\begin{prop}\label{FCGN}
Let $K$ be a normal subgroup of $G$; then $FC_G(K)\not=\{1\}$ if and only if one of the following conditions occurs:
\begin{itemize}
\item[(i)] $K$ contains a non-trivial finite subgroup normal in $G$,
\item[(ii)] $K$ contains a non-trivial free Abelian subgroup $\Z^n$ normal in $G$ such that the induced homomorphism $G\longrightarrow GL(n,\Z)$ has a finite image.
\end{itemize}
\end{prop}
\begin{proof}
Clearly if (i) or (ii) occurs then $FC_G(K)\not=\{1\}$. Let's show that (i) and (ii) are also necessary. Let $u\not=1$ lying in $FC_G(K)$ and $N_u$ the normal subgroup of $K$ finitely generated by all conjugates of $u$ in $G$. 
The centralizer $Z_G(N_u)$ of $N_u$ in $G$ has a finite index in $G$ since it equals the centralizer in $G$ of the finite family $\null^Gu$. Therefore the center $Z(N_u)=N_u\cap Z_G(N_u)$ of $N_u$ has finite index in $N_u$, and in particular $Z(N_u)$ is a finitely generated normal Abelian subgroup of $G$. Let ${\rm Tor}Z(N_u)$ be the subgroup of $Z(N_u)$ consisting of elements with finite order; ${\rm Tor}Z(N_u)$ is finite, and normal in $G$ since characteristic in $Z(N_u)$. If either $N_u$ is finite or ${\rm Tor}Z(N_u)\not=\{1\}$ then condition (i) holds. So we suppose in the following that $N_u$ is infinite and $Z(N_u)$ is torsion-free.
Under these hypotheses $Z(N_u)\simeq \Z^n$ for some $n>0$. The homomorphism:
$$\begin{array}{rcl}
G&\longrightarrow &Aut(Z(N_u)) \\
g&\longmapsto &\left(
v\longmapsto \null^gv
\right)
\end{array}
$$
composes with the isomorphism $Aut(Z(N_u))\simeq GL(n,\Z)$ into a homomorphism $\phi : G\longrightarrow GL(n,\Z)$. Given $g\in G$, the automorphism $x\mapsto\null^gx$ of $N_u$ permutes the finite generating set $\null^Gu$, and in particular $Aut(N_u)$ is finite as a subgroup of the symmetric group of $\null^Gu$. Since $Aut(N_u)$ projects onto, $\phi(G)$ is finite.
 Condition (ii) holds. 
\end{proof}

Now that we have a precise statement upon $FC_G(K)=\{1\}$, we use that condition to characterize the icc property in extensions of groups. As a first attempt the next result gives a sufficient condition, that turns to be non-necessary as  explained later. 

\begin{prop}\label{sufficient}
Let $G\not=\{1\}$ be a group that decomposes as an extension:
$$
1\longrightarrow K\longrightarrow G\overset{\proj}{\longrightarrow} Q\longrightarrow 1
$$
and $\Theta : Q\longrightarrow Out(K)$ the associated coupling. 
A sufficient condition for $G$ to be icc is:
\begin{itemize}
\item[(i)] $FC_G(K)=\{1\}$, and
\item[(ii)] $\Theta$ restricted to $FC(Q)$ is injective.
\end{itemize}
\end{prop}
\begin{proof}
Let $G$ decomposes as an extension as above. We suppose that $G$ is not icc and prove that one of the conditions (i) or (ii) fails.
Let $w\in G\setminus\{1\}$ with finite conjugacy class $\null^Gw$.\\
{\sl First case.} If $w\in K$, then $w\in FC_G(K)$ and condition (i) fails.\\
{\sl Second case.}  If $w\in G\setminus K$; note that $\proj(w)\in FC(Q)\setminus\{1\}$. Given $u\in K$ let $w_u=[u,w]$; on the one hand $w_u\in K$ and on the other hand $w_u\in FC(G)$.\\
{\sl First subcase.} If $\exists\,u\in K$ with $w_u=[u,w]\not=1$, then
$w_u\in FC_G(K)$ and condition (i) fails.\\
{\sl Second subcase.} If $\forall\, u\in K$, $[u,w]=1$, then $\Theta(\proj(w))=1$ and condition (ii) fails.
\end{proof}

\begin{cor} The icc property is stable by extension.
\end{cor}

Clearly condition (i) is also necessary for $G$ to be icc. Condition (ii) is not, as shown in the following example.
 
\begin{example}
\label{example_notnecessary}
We construct an example of a group $G$ that decomposes as an extension, and that is icc although condition (ii) of proposition \ref{sufficient} fails.

\noindent
Consider  the groups given by finite presentations and their subgroups:
$$K=< a_1,a_2,k_0,k_1|\,[a_1,a_2],[a_i,k_j],i,j=1,2>\ \simeq (\ZZ)\times F_2$$
$$A\simeq\ZZ\ \text{the subgroup of $K$ generated by $a_1$,$a_2$,}$$ $$F_K\simeq F_2\ \text{the free subgroup of $K$ generated by $k_0,k_1$,}$$
$$Q=< q_0,q_1,q_2|\,[q_0,q_1]=[q_0,q_2]=1>\ \simeq\Z\times F_2$$
$$Q_0\ \text{the cyclic subgroup of $Q$ generated by $q_0$,}$$
$$F_Q\simeq F_2\ \text{the free subgroup of $Q$ generated by $q_1,q_2$~.}$$

Let $\phi,\,\psi\in SL(2,\Z)$ be the automorphisms of $A$ whose matrices with respect to the basis $a_1,a_2$ are:
$$
M_\phi=\left(\begin{matrix}
1&1\\
0&1
\end{matrix}
\right)\qquad
 M_\psi=\left(\begin{matrix}
1&0\\
1&1
\end{matrix}
\right)
$$
Consider the three automorphisms 
$\theta(q_0),\theta(q_1),\theta(q_2)\in Aut(K)$ defined by:\\ \indent$\forall\,k\in K,\, \forall\ a\in A$,
$$
\theta(q_0)(k)=\null^{k_0}k,\ \ \null\  
\left\lbrace\begin{array}{l}
\theta(q_1)(a)=\phi(a) \\
\theta(q_1)(k_0)=k_0a_2 \\
\theta(q_1)(k_1)=k_1
\end{array}\right.
,\ \ \null\  
\left\lbrace\begin{array}{l}
\theta(q_2)(a)=\psi(a) \\
\theta(q_2)(k_0)=k_0a_1 \\
\theta(q_2)(k_1)=k_1
\end{array}\right.
$$ 
One verifies that $\theta(q_0)$ both commutes with $\theta(q_1)$ and $\theta(q_2)$ so that $\theta$ extends to $\theta:Q\longrightarrow Aut(K)$.
Define $G=K\rtimes_{\theta}Q$;  $G$ is a split extension with coupling $\Theta=\Pi\circ\theta$ where $\Pi:Aut(K)\rightarrow Out(K)$ is the natural projection.\smallskip

\noindent
Here $FC(Q)=Q_0$ and  $\Theta$ is non injective when restricted to $Q_0$, indeed $\Theta(q_0)=1$ for $\theta(q_0)$ is inner. Therefore $G$ does not satisfy condition (ii) of proposition \ref{sufficient}. Let's prove that nevertheless $G$ is icc.\smallskip

\noindent
Note that $A$ is the center of $K$, so that $\theta$ induces $\ol{\theta}:Q\longrightarrow Aut(A)$; one has 
$\ol{\theta}(q_0)=Id$, $\ol{\theta}(q_1)=\phi$ and $\ol{\theta}(q_2)=\psi$. 
Moreover $FC(K)=A$ is free Abelian with rank 2. The maximal proper subgroup of $A$ 
preserved by $\phi$ (respectively $\psi$) is cyclic generated by $a_1$ (respectively $a_2$); hence  $\ol{\theta}(Q)$ does not preserve a proper  subgroup $\not=\{1\}$.
Moreover $\ol{\theta}(Q)$  is infinite, as well as its isomorphic image in $GL(2,\Z)$. According to proposition \ref{FCGN}, $FC_G(K)=\{1\}$. Hence for all $k\in K$, $\null^Gk$ is infinite, and it suffices now to prove that all elements in $G\setminus K$ have infinite conjugacy classes.\smallskip

\noindent 
An element in $G\setminus K$ is of the type: 
$w=akq_0^nq$ with $a\in A$, $k\in F_K$, $q\in F_Q$ and $n\in\Z^\ast$. Whenever $q\not=1$, $\null^Gw$ is infinite since it projects onto the infinite conjugacy class $\null^{F_Q}q$ in $F_Q$ free group of rank 2.  Consider now the remaining case $w=akq_0^n$. If $k$ does not lie in the subgroup generated by $k_0$, then $k_0^pwk_0^{-p}=ak_0^pkk_0^{-p}q_0^n$ and $\{k_0^pwk_0^{-p}\,;\,p\in\Z\}\subset \null^Gw$ is infinite.  
Hence we are now left with the case $w=ak_0^mq_0^n$, $m\in\Z,n\in\Z^*$.
Consider the set $\{q_1^pwq_1^{-p}\,;\,p\in\Z\}\subset\null^Gw$,
$$
q_1^pwq_1^{-p}=q_1^pak_0^mq_0^nq_1^{-p}=\phi^p(a)k_0^ma_2^m\phi(a_2^m)\cdots\phi^{p-1}(a_2^m),
$$
and one obtains that $\{q_1^pwq_1^{-p}\,;\,p\in\Z\}$ is finite if and only if:
\begin{equation}
\exists\,p\in\N,\ \phi^p(a)^{-1}a=a_2^m\phi(a_2^m)\cdots\phi^{p-1}(a_2^m)~.
\end{equation}
We now use additive notations and discuss on equation (1):
$$
\exists\,p\in\N,\ (Id-\phi^p)(a)=(Id+\phi+\cdots+\phi^{p-1})(a_2^m)~.
$$
One has:
$$
M_{\phi^p}=\left(\begin{matrix}1&p\\0&1\end{matrix}\right)\quad;\quad
\mathbb{I}-M_{\phi^p}=\left(\begin{matrix}0&-p\\0&0\end{matrix}\right)\quad;\quad
\mathbb{I}+M_\phi+\cdots+M_{\phi^{p-1}}=
\left(\begin{matrix}p&\frac{p(p-1)}{2}\\0&p\end{matrix}\right),
$$
and $a_2^m=\left(\begin{matrix} 0\\m\end{matrix}\right)$; set $a=\left(\begin{matrix} x\\ y\end{matrix}\right)$, the equation (1) becomes:
$$
\exists\,p\in\N,\quad \left(\begin{matrix}-py\\0\end{matrix}\right)
=\left(\begin{matrix} m\frac{p(p-1)}{2}\\ mp\end{matrix}\right)
$$
which should hold only for $m=0$. But when $m=0$, $w=aq_0^n$, on the one hand $q_1^pwq_1^{-p}=\phi^p(a)q_0^n$ and $\{q_1^pwq_1^{-p}\,;\,p\in\Z\}$
is infinite whenever $a$ and $a_2$ are not collinears; on the other hand  $q_2^pwq_2^{-p}=\phi^p(a)q_0^n$ and $\{q_2^pwq_2^{-p}\,;\,p\in\Z\}$
is infinite whenever $a$ and $a_1$ are not collinears. In conclusion $\null^Gw$ is infinite, and this proves that $G$ is icc.

\end{example}

\section{\bf General results on extensions of groups}\label{partie:extension}

We establish a necessary and sufficient condition for an arbitrary group $G$ that decomposes as an extension to be icc. Keeping in mind the sufficient conditions given in proposition \ref{sufficient}, this is done by weakening the condition (ii). 
\\

\noindent
We focus on a group $G\not=\{1\}$ which decomposes as an extension, {\it i.e.} that fits into a short exact sequence:
\[
\xymatrix{
1\ar[r]&K\ar[r]&G\ar^{\proj}[r]  & Q\ar[r] & 1
 }
\]
with $\Theta:Q\longrightarrow Out(K)$, the associated coupling. 
Throughout the section, $Z$ denotes the center $Z(K)$ of $K$ and for any $q\in Q$, $C(q)$ denotes the centralizer $C_Q(q)$ of $q$ in $Q$. 

\subsection{A preliminary definition}\label{definitionH1}
We fix a section $s:Q\longrightarrow G$, {\it i.e.} a map such that $\forall q\in Q, \proj\circ s(q)=q$\,; we will rather denote $\overline{q}:=s(q)$. 
The section $s$ defines $\forall q\in Q$ a lift $\theta_q$ of $\Theta(q)$ in $Aut(K)$ defined by $\theta_q(x)=\ol{q}\,x\,\ol{q}^{-1}=\up{x}{\ol{q}}$ for all $x\in K$ (for more convenience, we shall use both notations $\theta_q(x)$ and $\up{x}{\ol{q}}$). We will also write $\theta_a(x)=\up{x}{a}$ for all $a,x\in K$. \smallskip

Let $q\in \ker \Theta$; {\it i.e.} $q$ lies in $Q$ and there exists $k\in K$ (depending on $s$) such that $\theta_q(x)=k\,x\,k^{-1}$ for all $x\in K$. given $u\in C(q)$, $uqu^{-1}=q$ in $Q$; hence  there exists an element $\delta_q(u)$ in $K$ (depending on $s$), defined by:
$$\ol{u}\,\ol{q}\,\ol{u}^{-1}=\ol{q}\,\delta_q(u)\quad \text{in $G$}~.$$

Keeping these notations in mind, the next proposition asserts that each $q\in\ker\Theta$ defines an element $[q]$ in the first cohomology group $H^1(C(q),Z)$; it will be proved in section \S \ref{souspartie:extension,preuve}.\smallskip

\begin{propdef}\label{H1ext}
A section $s:Q\longrightarrow G$ is given. Let $q\in\ker\Theta$ and $k\in K$ be such that $\forall\, x\in K$, $\theta_q(x)=kxk^{-1}$.

\begin{itemize}
\item[(1)] For any $u\in C(q)$ let $\delta_q(u)\in K$ be such that $\ol{u}\,\ol{q}\,\ol{u}^{-1}=\ol{q}\,\delta_q(u)$. Then:\\
 the element $d_q(u):=\delta_q(u)^{-1}k^{-1}\,\up{k}{\ol{u}}$ lies in the center $Z$ of $K$.\smallskip
 \item[(2)] Define the map:
$$\begin{array}{cccl}d_q: 
&C(q)&\longrightarrow &Z\\
&u&\longrightarrow & d_q(u):=\delta_q(u)^{-1}k^{-1}\,\up{k}{\ol{u}}~.
\end{array}
$$
The map $d_q$ is a 1-cocycle.\smallskip
\item[(3)] The cocycle $d_q:C(q)\longrightarrow Z$ defines an element $[q]$ in $H^1(C(q),Z)$ that only depends on $q\in\ker\Theta$ and on the equivalence class of the extension $G$ of $K$ by $Q$.
\end{itemize}
\end{propdef}

The first cohomology group $H^1(C(q),Z)$ is the quotient group of the Abelian group of {\sl crossed homomorphisms} with respect to multiplication:
$$
f:C(q)\longrightarrow Z,\quad\text{such that}\ \forall\,u,v\in C(q),\
f(uv)=f(v)\left(\null^{\ol{v}}f(u)\right)~. 
$$
by the normal subgroup of those crossed homomorphisms that are {\sl principal}:
$$
f:C(q)\longrightarrow Z,\quad\text{such that}\ \exists\,z\in Z,\,\forall\,u\in C(q),\ f(u)=z^{-1}\,\null^{\ol{u}}z~.
$$

\begin{remark}
\label{H1split} In case the extension splits, $G=K\rtimes Q$:\smallskip\\
For any  $q\in\ker\Theta$,  the homology class $[q]\in H^1( C(q),Z)$ given by proposition-definition \ref{H1ext} is:
let $\theta_q(x)=k\,x\,k^{-1}$ for $k\in K$, then $[q]$ is represented by the 1-cocycle:
$$\begin{array}{ccc} C(q)&\longrightarrow &Z\\
u&\longrightarrow &[k^{-1},u]
\end{array}
$$ 
and does not depend on the choice of $k\in K$ such that $\theta_q(x)=\up{x}{k}$.
%
\medskip
\end{remark}

\subsection{Statement of the main result}
We can now enunciate a necessary and sufficient condition for a group $G$ which decomposes as an extension to be icc. (The proof is achieved in \S \ref{proofextension}.)

\begin{thm}[icc extension] \label{extension}
Let $G\not=\{1\}$ be a group that decomposes as an extension:
\[
\xymatrix{
1\ar[r]&K\ar[r]&G\ar^{\proj}[r]  & Q\ar[r] & 1
 }
\]
Let $\Phi:FC(Q)\longrightarrow Out(K)$ denotes the restriction of the coupling $\Theta$ to $FC(Q)$.\smallskip

A necessary and sufficient condition for $G$ to be icc is:
\begin{itemize}
\item[(i)] $FC_G(K)=\{1\}$, and
 \item[(ii)] $\ker \Phi$ does not contain an element $q\not=1$ such that $[q]=0$ in $H^1(C(q),Z)$.
\end{itemize}
\end{thm}

\medskip


\begin{example}
\label{examplescenterless} In Example \ref{example_notnecessary}, we have exhibited a group $G$ which is icc while holds condition (i) but not condition (ii) of proposition \ref{sufficient}. Let's verify that $G$ satisfies condition (ii) of Theorem \ref{extension}.\smallskip\\
One has $Z=A$, $C(q)=Q$, $FC(Q)=Q_0$ and $\ker\Phi=Q_0$ is cyclic generated by $q_0$; let $q\not=1\in \ker\Phi$, 
$q=q_0^n$, hence $[q]\in H^1(Q,A)$ is represented by the 1-cocycle:
$$\begin{array}{rl}
 d_q:&Q\longrightarrow A\\
 &u\longmapsto [k_0^{-n},u]
 \end{array}
 $$
  (cf. remark \ref{H1split}); one has $d_q(q_1)=a_2^n$.
Suppose that $[q]=0$, therefore $\exists\,a\in A$ such that $\forall\,u\in Q$, $d_q(u)=a^{-1}\,\null^{{u}}a$. By taking $u=q_1$ one obtains $d_q(q_1)=a^{-1}\phi(a)$. We now use additive notations in the basis $a_1,a_2$ of $A$. If $[q]=0$ there exists $a=(x,y)$ in $A$ such that:
$$
(\phi-Id)(a)=na_2\quad\iff\quad \left(\begin{matrix}0&1\\0&0\end{matrix}\right)\left(\begin{matrix}x\\y
\end{matrix}\right)=\left(\begin{matrix}y\\0\end{matrix}\right)
=\left(\begin{matrix}0\\n\end{matrix}\right)
$$
which is impossible. Therefore for any $q\in \ker\Phi\setminus\{1\}$, $[q]\not=0$; the condition (ii) of Theorem \ref{extension} holds.
\end{example}

\begin{example}
\label{examplesextension}
Additional hypotheses on the groups involved can make the result more concise. We emphasize here some examples (see also \S \ref{pc}):
\begin{itemize}
\item[--] {\sl whenever $K$ is centerless, $G$ is icc if and only if condition (i) holds and:\\ $\Phi :FC(Q)\longrightarrow Out(K)$ is injective, }
\end{itemize}
in particular:
\begin{itemize}
\item[--] {\sl whenever $K$ is icc,  $G$ is icc if and only if $\Phi :FC(Q)\longrightarrow Out(K)$ is injective, }
\end{itemize}
with Proposition \ref{simplecci}, 
\begin{itemize}
\item[--] {\sl whenever $K$ is simple,  $G$ is icc if and only if $K$ is infinite and $\Phi :FC(Q)\longrightarrow Out(K)$ is injective. }
\end{itemize}
Moreover, if $K$ is centerless the extension is characterized up to equivalence by the coupling $\Theta:Q\longrightarrow Out(K)$ (cf. Corollary 6.8, Chap IV, \cite{brown}); it follows that:
\begin{itemize}
\item[--] {\sl whenever $K$ is centerless and $Q$ is simple, $G$ is icc if and only if $K$ is icc and either $Q$ is infinite or the extension is not equivalent to $K\times Q$.}
\end{itemize}
\end{example}

\subsection{A noteworthy particular case}
\label{fcqfg2}
(Results in this section are proved in \S \ref{preuvehomomorphism}).
The 1-cocycle associated to an element in $\ker \Phi$, $\Phi : FC(Q)\longrightarrow Out(K)$, defined in proposition \ref{H1ext} yields also a homomorphism from $\ker\Phi$ to $H^1(C_Q(FC(Q),Z)$. 

\begin{propdef}\label{defhomomorphism}
There exists a homomorphism 
$$\Xi:\ker\Phi\longrightarrow H^1(C_Q(FC(Q)),Z)~,$$
 defined by:
for all $q\in \ker\Phi$, $\Xi(q)$ is represented by the 1-cocycle $d_q$:
$$\begin{array}{crcl}d_q: 
&C_Q(FC(Q))&\longrightarrow &Z\\
&u&\longrightarrow & d_q(u)=\delta_q(u)^{-1}k^{-1}\,\up{k}{\ol{u}\,}\end{array}
$$ 
where $\theta_q(x)=k\,x\,k^{-1}$ for all $x\in K$.
\end{propdef}

\noindent Under an additional hypothesis on $Q$, namely that $FC(Q)$ 
is finitely generated, one gives a more concise necessary and sufficient condition where condition (ii) is rephrased by mean of a homomorphism from $\ker \Phi$ to $H^1(C_Q(FC(Q)),Z(K))$.%
\vskip 0.2cm

\begin{prop}[In case $FC(Q)$ is finitely generated]\label{extfcqfg}
Under the same hypotheses as Theorem \ref{extension}, if moreover $FC(Q)$ is finitely generated, or more generally if $C_Q(FC(Q))$ has finite index in $Q$,a necessary and sufficient condition for $G$ to be icc is:
\begin{itemize}
\item[(i)] $FC_G(K)=\{1\}$, and
\item[(ii$'$)]  The homomorphism $\Xi:\ker\Phi\longrightarrow H^1(C_Q(FC(Q)),Z)$ is injective.
\end{itemize}
\end{prop}

\noindent
An example of the wide applicability of this formulation is given in  example \ref{nonperfect} of \S \ref{section:qabelien}. The next example shows that the hypothesis  $[Q:C_Q(FC(Q)]<\infty$ is necessary in Proposition \ref{extfcqfg}.
\vskip 0.2cm

\begin{example}
   Here is an example where $C_Q(FC(Q))$ has infinite index in $Q$ and $G$ is icc while condition (ii$'$) fails.
Consider the groups:
\begin{gather*}
A=\ <a_1,\ldots, a_n,\ldots\,|\, [a_i,a_j],\,\forall\, i,j\in\N^*>\ \simeq \Z^\w\\
Q=\ <A,t\,|\,\forall\,n\in\N^*,\, t^na_n\,t^{-n}=a_n>
\end{gather*}
and
\begin{gather*}
K=<\a,\b,\a_1,\b_1,\ldots,\a_n,\b_n,\ldots\,|\qquad\qquad\qquad\qquad\qquad\qquad\qquad\qquad\qquad\qquad\qquad\qquad\\
\forall\, i,j\in\N^*,\,[\a_i,\b_j]=[\a,\a_i]=[\a,\b_j]=[\b,\a_i]=[\b,\b_j]=1>\ \simeq F_2\times\bigoplus_{n\in \N^*}(\Z\oplus\Z)
\end{gather*}
Let $A_n$ be the Abelian subgroup of $K$ generated by $\a_n,\b_n$. Consider an anosov automorphism ({\it i.e.} with irrational eigenvals) $\varphi\in SL(2,\Z)$ of $\Z\oplus\Z$ and define the homomorphism $\theta:A\longrightarrow Aut(K)$ by: 
$$\forall\,x\in K,\ \theta(a_1)(x)=\up{x}{\a}\ \text{and}\ \forall\, n\in\N^*,
\theta(a_{n+1})(x)=\begin{cases}
\varphi(x) & \text{if $x\in A_{n}$}\\
x & \text{otherwise}
\end{cases}
$$
 Let $G_0=K\rtimes_\theta A$; one has $FC(K)=Z(K)=\bigoplus_{n\in\N^*}A_n$ and $G_0$ is not icc ($\a^{-1}\a_1^{G_0}$ is finite). Using Theorem \ref{extension} and Proposition \ref{extfcqfg} one sees  that 
the homomorphism $\Xi:\ker\Phi\longrightarrow H^1(A,Z(K))$ is non-injective (here $\ker\Phi$ is cyclic generated by $a_1$ and $\Xi(\ker\Phi)=\{0\}$).\smallskip\\
Extend $\theta$ to $\vartheta:Q\longrightarrow Aut(K)$ by setting $\vartheta(t)(\a)=\a\a_1$ and $\vartheta(t)(x)=x$ for $x=\b,\a_1,\b_1,\ldots,\a_n,\b_n$. 
 Define $G=K\rtimes_{\vartheta} Q$;  $G$ is icc ($\vartheta(Q)(\a^n)$ is infinite); nevertheless, $\Xi:\ker\Phi\longrightarrow H^1(C_Q(FC(Q),Z(K))$ is non-injective (here again $\ker\Phi$ is generated by $a_1$ and $C_Q(FC(Q))=A$).
\end{example}

\subsection{Proof of the results}\label{souspartie:extension,preuve}
%
\subsubsection{Proof of Proposition-Definition \ref{H1ext}}\label{defH1}
In the following, we consider elements $q\in \ker \Theta$, $k\in K$ such that $\theta_q(x)=k\,x\,k^{-1}$ for all $x\in K$, and for all $u\in C(q)$, $\delta_q(u)$ in $K$ defined by:
$$\ol{u}\,\ol{q}\,\ol{u}^{-1}=\ol{q}\,\delta_q(u)\quad \text{in $G$}~.$$

\begin{lem}\label{cocyclecentre}
For all $u\in C(q)$,  $\delta_q(u)^{-1}k^{-1}\,\up{k}{\ol{u}}$ lies in the center $Z$ of $K$.
\end{lem}
\begin{proof}
Since $u\in C(q)$, $\ol{u}\,\ol{q}\,\ol{u}^{-1}=\ol{q}\,\delta_q(u)$  one has $\theta_u\circ\theta_q\circ\theta_u^{-1}=\theta_q\circ\theta_{\delta_q(u)}$ in $Aut(K)$. Hence $\forall\,x\in K$,
$$
\begin{array}{cc}
&\theta_u\circ\theta_q\circ\theta_u^{-1}(x)=\theta_q\circ\theta_{\delta_q(u)}(x)\\
\iff\quad &\theta_u(k\,\theta_u^{-1}(x)\,k^{-1})=k\,\delta_q(u)\,x\,\delta_q(u)^{-1}\,k^{-1}\\
\iff\quad &\theta_u(k)\,x\,\theta_u(k)^{-1}=k\,\delta_q(u)\,x\,\delta_q(u)^{-1}\,k^{-1}
\end{array}
$$
and it follows that $\delta_q(u)^{-1}k^{-1}\,\up{k}{\ol{u}}=\delta_q(u)^{-1}k^{-1}\theta_u(k)$ commutes with all $x\in K$.
\end{proof}

Given $q\in \ker\Theta$ and $k\in K$ as above, define the map:
$$\begin{array}{cccl}d_q: 
&C(q)&\longrightarrow &Z\\
&u&\longrightarrow & d_q(u):=\delta_q(u)^{-1}k^{-1}\,\up{k}{\ol{u}}~.
\end{array}
$$
\noindent
Note that {\it a-priori} $d_q$  does  depend  on both the section $s$ and  the choice of an element $k\in K$ such that $\theta_q(x)=\up{x}{k}$. 

\begin{lem}
The map $d_q:C(q)\longrightarrow Z$ is a 1-cocycle.
\end{lem}

\begin{proof}
We need to show that $\forall\, u,v\in C(q),\, d_q(uv)=d_q(u)\,\up{d_q(v)}{\ol{u}}$. We denote by $f:Q\times Q\longrightarrow K$ the 2-cocycle defined by the extension and by the section $s$, such that: 
$$
\forall\,x,y\in Q,\quad\ol{x}\,\ol{y}=\ol{xy}\,f(x,y)~.
$$
 Let $u,v\in C(q)$.
In order to compute $d_q(uv)$ we need to compute first $\delta_q(uv)$ and $\theta_{uv}$. Since $\ol{uv}=\ol{u}\,\ol{v}\,f(u,v)^{-1}$, one has $\theta_{uv}=\theta_u\circ\theta_v\circ \theta_{f(u,v)^{-1}}$.\\
{\it Computation of $\delta_{q}(uv)$.} One has $\ol{uv}\,\ol{q}\,\ol{uv}^{-1}=\ol{q}\,\delta_q(uv)$.
\[
\begin{array}{rl}
\ol{uv}\,\ol{q}\,\ol{uv}^{-1}
&=\ol{u}\,\ol{v}\,f(u,v)^{-1}\ol{q}\,f(u,v)\,\ol{v}^{-1}\ol{u}^{-1}\\
&= \ol{u}\,\up{f(u,v)^{-1}}{\ol{v}}\,\ol{v}\,\ol{q}\,\ol{v}^{-1}\,\up{f(u,v)}{\ol{v}}\,\ol{u}^{-1}\\
&= \up{f(u,v)^{-1}}{\ol{u}\,\ol{v}}\,\ol{u}\,\ol{q}\,\delta_q(v)\,\ol{u}^{-1}\,\up{f(u,v)}{\ol{u}\,\ol{v}}\\
&=  \up{f(u,v)^{-1}}{\ol{u}\,\ol{v}}\,\ol{u}\,\ol{q}\,\ol{u}^{-1}\,\up{\delta_q(v)}{\ol{u}}\,\up{f(u,v)}{\ol{u}\,\ol{v}}\\
&=  \up{f(u,v)^{-1}}{\ol{u}\,\ol{v}}\,\ol{q}\,\delta_q(u)\,\up{\delta_q(v)}{\ol{u}}\,\up{f(u,v)}{\ol{u}\,\ol{v}}\\
&=\ol{q}\,\up{f(u,v)^{-1}}{\ol{q}^{-1}\ol{u}\,\ol{v}}\,\delta_q(u)\,\up{\delta_q(v)}{\ol{u}}\,\up{f(u,v)}{\ol{u}\,\ol{v}}
\end{array}
\]
Therefore  $\delta_q(uv)=\up{f(u,v)^{-1}}{\ol{q}^{-1}\ol{u}\,\ol{v}}\,\delta_q(u)\,\up{\delta_q(v)}{\ol{u}}\,\up{f(u,v)}{\ol{u}\,\ol{v}}$.\\
 By applying the definition, we obtain on the one hand:
\[
\begin{array}{rl}
d_q(uv)&=\delta_q(uv)^{-1}k^{-1}\theta_{uv}(k)\\
&=\up{f(u,v)^{-1}}{\ol{u}\,\ol{v}}\,\up{\delta_q(v)^{-1}}{\ol{u}}\,\delta_q(u)^{-1}\,\up{f(u,v)}{\ol{q}^{-1}\ol{u}\,\ol{v}}\,
k^{-1}\,\up{(f(u,v)^{-1}k\,f(u,v))}{\ol{u}\,\ol{v}}\\
&=\up{f(u,v)^{-1}}{\ol{u}\,\ol{v}}\,\up{\delta_q(v)^{-1}}{\ol{u}}\,\delta_q(u)^{-1}\,\up{f(u,v)}{\ol{q}^{-1}\ol{u}\,\ol{v}}\,
\up{f(u,v)^{-1}}{\ol{q}^{-1}\ol{u}\,\ol{v}}\,
k^{-1}\,\up{(k\,f(u,v))}{\ol{u}\,\ol{v}}\\
&= \up{f(u,v)^{-1}}{\ol{u}\,\ol{v}}\,\up{\delta_q(v)^{-1}}{\ol{u}}\,\delta_q(u)^{-1}
k^{-1}\,\up{(k\,f(u,v))}{\ol{u}\,\ol{v}}\\
&= \up{f(u,v)^{-1}}{\ol{u}\,\ol{v}}\,\up{\delta_q(v)^{-1}}{\ol{u}}\,\delta_q(u)^{-1}
k^{-1}\,\up{k}{\ol{u}\,\ol{v}}\,\up{f(u,v)}{\ol{u}\,\ol{v}}\\
\end{array}
\]
and since $d_q(uv)$ lies in the center $Z$ of $K$ (lemma \ref{cocyclecentre}), $f(u,v)$, $\delta_q(v)$ both lie in $K$, and $K$ is preserved both by $\theta_u$, $\theta_v$:
$$
\begin{array}{rl}
d_q(uv)&=\up{\delta_q(v)^{-1}}{\ol{u}}\,\delta_q(u)^{-1}
k^{-1}\,\up{k}{\ol{u}\,\ol{v}}\\
&= \delta_q(u)^{-1}
k^{-1}\,\up{k}{\ol{u}\,\ol{v}}\,\up{\delta_q(v)^{-1}}{\ol{u}}~.
\end{array}
$$
On the other hand, since $\delta_q(v)\in K$, $d_q(v)\in Z$ and $Z$ is a characteristic subgroup of $K$:
\[
\begin{array}{rl}
\up{d_q(v)}{\ol{u}} &= \up{(\delta_q(v)^{-1}k^{-1}\,\up{k}{\ol{v}})}{\ol{u}}\\
&=\up{(k^{-1}\,\up{k}{\ol{v}}\,\delta_q(v)^{-1})}{\ol{u}}
\end{array}
\]
so that:
$$\begin{array}{rl}
d_q(u)\,\up{d_q(v)}{\ol{u}}&=\delta_q(u)^{-1}k^{-1}\,\up{k}{\ol{u}}\,\up{k^{-1}}{\ol{u}}\,\up{k}{\ol{u}\,\ol{v}}\,\up{\delta_q(v)^{-1}}{\ol{u}}\\
&=\delta_q(u)^{-1}k^{-1}\,\up{k}{\ol{u}\,\ol{v}}\,\up{\delta_q(v)^{-1}}{\ol{u}}\\
&=d_q(uv)
\end{array}
$$
which proves that $d_q:C(k)\longrightarrow Z$ is a 1-cocycle. 
\end{proof}

\begin{lem}
The cocycle $d_q:C(q)\longrightarrow Z$ defines an element $[q]$ in $H^1(C(q),Z))$ which only depends on  $q\in \ker \Theta$ and on the equivalence class of the extension $G$ of $K$ by $Q$.
\end{lem}

\begin{proof}
The cocycle $d_q$ defines an element of $H^1(C(q),Z)$, we need to prove that it depends neither on $k\in kZ$ nor on the section $s$.

Let $k_1\in kZ$, say $k_1=kz$ for some $z\in Z$, which, as above, defines the 1-cocycle $d'_q:C(q)\longrightarrow Z$ by:
\[
\begin{array}{rl}
d'_q(u)&=\delta_q(u)^{-1}k_1^{-1}\,\up{k_1}{\ol{u}}\\
&=\delta_q(u)^{-1}z^{-1}k^{-1}\,\up{k}{\ol{u}}\,\up{z}{\ol{u}}\\
&=\delta_q(u)^{-1}k^{-1}\,\up{k}{\ol{u}}\,z^{-1}\,\up{z}{\ol{u}}\\
&=d_q(u)\,z^{-1}\,\up{z}{\ol{u}}
\end{array}
\]
$d_q$, $d'_q$ differ by the 1-coboundary : $z\longrightarrow z^{-1}\,\up{z}{\ol{u}}$, so that they define the same element of $H^1(C(q),Z)$.\smallskip

Now we show that $[q]$ does not depend on the section $s$. We proceed in two steps. First we change $s(q)=\ol{q}$ into $s(q)=\widetilde{q}=\ol{q}\,\widehat{q}$ for some $\widehat{q}\in K$. Consider $k'=k\,\widehat{q}$; since for all $x\in K$, $\up{x}{\widetilde{q}}=\up{x}{k'}$ it defines a 1-cocycle denoted ${d}'_q$ whose class in $H^1(C(q),Z)$ does not depend on the choice of $\widehat{q}$. Obviously, ${d}'_q(q)=d_q(q)$.
 Given $u\in Z(q)$, $\delta_q(u)$  changes into $\delta'_q(u)=\widehat{q}^{-1}\delta_q(u)\,\theta_u(\widehat{q})$, for:
\[
\ol{u}\,\widetilde{q}\,\ol{u}^{-1}=\ol{u}\,\ol{q}\,\widehat{q}\,\ol{u}^{-1}=\ol{u}\,\ol{q}\,\ol{u}^{-1}\,\up{\widehat{q}}{\ol{u}\,}=\ol{q}\,\delta_q(u)\,\up{\widehat{q}}{\ol{u}\,}=\widetilde{q}\,\widehat{q\,}^{-1}\,\delta_q(u)\,\up{\widehat{q}}{\ol{u}\,}
\]
Hence:
\[
d'_q(u)=\delta'_q(u)^{-1}{k'}^{-1}\theta_u(k')=\up{\widehat{q}^{-1}}{\ol{u}\,}\,\delta_q(u)^{-1}\widehat{q}\,\widehat{q}^{-1}k^{-1}\,\up{k}{\ol{u}}\,\up{\widehat{q}}{\ol{u}\,}=\up{\widehat{q}^{-1}}{\ol{u}\,}\,\delta_q(u)^{-1}k^{-1}\,\up{k}{\ol{u}}\,\up{\widehat{q}}{\ol{u}\,}
\]
and since $d'_q(u)\in Z$ and $\up{\widehat{q}}{\ol{u}\,}\in K$, one obtains $d'_q(u)=\delta_q(u)^{-1}k^{-1}\,\up{k}{\ol{u}\,}=d_q(u)$, which achieves the first step.

Now we change for $u\in Z(q)$, $s(u)=\ol{u}$ into $s(u)=\widetilde{u}=\ol{u}\,\widehat{u}$ for some $\widehat{u}\in K$.
It changes $\delta_q(u)$ into $\delta'_q(u)=\delta_q(u)\,\up{(\up{\widehat{u}}{\ol{q}^{-1}}\widehat{u}^{-1})}{\ol{u}}$ for:
$$
\widetilde{u}\,\ol{q}\,\widetilde{u}^{-1}=\ol{u}\,\widehat{u}\,\ol{q}\,\widehat{u}^{-1}\,\ol{u}^{-1}=\ol{u}\,\ol{q}\,\up{\widehat{u}}{\ol{q}^{-1}}\,\widehat{u}^{-1}\,\ol{u}^{-1}=\ol{u}\,\ol{q}\,\ol{u}^{-1}\,\up{(\up{\widehat{u}}{\ol{q}^{-1}}\widehat{u}^{-1})}{\ol{u}}=\ol{q}\,\delta_q(u)\,\up{(\up{\widehat{u}}{\ol{q}^{-1}}\widehat{u}^{-1})}{\ol{u}}~.
$$
Moreover, $\up{k}{\widetilde{u}\,}=\up{k}{\ol{u}\,}\up{(\up{\widehat{u}}{\ol{q}^{-1}}\widehat{u}^{-1})}{\ol{u}}$ for:
$$
\up{k}{\widetilde{u}}=\ol{u}\,\widehat{u}\,k\,\widehat{u}^{-1}\,\ol{u}^{-1}=\ol{u}\,k\,\up{\widehat{u}}{k^{-1}}\,\widehat{u}^{-1}\ol{u}^{-1}=\ol{u}\,k\,\ol{u}^{-1}\,\up{(\up{\widehat{u}}{{k}^{-1}}\widehat{u}^{-1})}{\ol{u}\,}=\up{k}{\ol{u}\,}\,\up{(\up{\widehat{u}}{\ol{q}^{-1}}\widehat{u}^{-1})}{\ol{u}}
$$
and we obtain:
$$
\begin{array}{rl}
d'_q(u)&=\delta'_q(u)^{-1}{k}^{-1}\,\up{k}{\widetilde{u}\,}\\
&=(\up{(\up{\widehat{u}}{\ol{q}^{-1}}\widehat{u}^{-1})}{\ol{u}\,})^{-1}\delta_q(u)^{-1}{k}^{-1}\,\up{k}{\ol{u}\,}\up{(\up{\widehat{u}}{\ol{q}^{-1}}\widehat{u}^{-1})}{\ol{u}\,}\\
&=\up{(\up{\widehat{u}}{\ol{q}^{-1}}\widehat{u}^{-1}}{\ol{u}\,})^{-1}\,d_q(u)\,\up{(\up{\widehat{u}}{\ol{q}^{-1}}\widehat{u}^{-1})}{\ol{u}\,}\\
&=d_q(u)
\end{array}
$$
since $d_q(u)\in Z$ and $\up{(\up{\widehat{u}}{\ol{q}^{-1}}\widehat{u}^{-1})}{\ol{u}\,}\in K$, this achieves the second step.
\end{proof}

\subsubsection{Proof of Theorem \ref{extension}}\label{proofextension}
We already know that condition (i) is necessary (see \S \ref{section1}). We now proceed in two steps.\medskip\\
{\it Step 1. $G$ icc $\implies$ condition {\rm (ii)}.}\\
Suppose on the contrary that condition (ii) fails: there exists $q \in \ker\Phi$ such that $q\not=1$ and $[q]=0$ in $H^1(C_Q(q),Z(K))$.
According to the following Lemma \ref{lastlemma} there exists $\w\in G\setminus K$ such that $C_G(\w)$ contains the preimage $\pi^{-1}(C_Q(q))$. But since $q\in FC(Q)$, $C_Q(q)$ has finite index in $Q$ so that, $C_G(\w)$ also has finite index in $G$. Therefore $G$ is not icc.
\hfill$\square$\\

\begin{lem}\label{lastlemma}
If there exists $q \in \ker\Phi\setminus\{1\}$ such that $[q]=0$ in $H^1(C_Q(q),Z(K))$, then there exists $\w\in G\setminus K$ such that $C_G(\w)\supset \pi^{-1}(C_Q(q))$.
\end{lem}

\noindent{\it Proof of Lemma \ref{lastlemma}.}
 We fix a section $s:Q\longrightarrow G$ and use the same notations as in \S \ref{definitionH1}. Let $k\in K$ such that $\theta_q(x)=k\,x\,k^{-1}$ for all $x\in K$. Since $[q]=0$, there exists $z\in Z(K)$ such that for any $u\in C_Q(q)$, 
$$
d_q(u):=\delta_q(u)^{-1}k^{-1}\,\up{k}{\ol{u}\,}=z^{-1}\,\up{z}{\ol{u}\,}~.
$$
 Let $\w=\ol{q}\,z\,k^{-1}\in G$; since $q\not=1$, $\w\in G\setminus K$. Its centralizer $C_G(\w)$ contains $K$; indeed for all $x$ in $K$,
$$ \begin{array}{rl}
\w\, x\,\w^{-1}&= \ol{q}\,z\,k^{-1} x\, k\, z^{-1}\,\ol{q}^{-1}\\
&= \ol{q}\,k^{-1} x\, k\, \ol{q}^{-1}\\
&= k\,k^{-1}x\,k\,k^{-1}\\
&=x
\end{array}
$$
Moreover 
  for any $u\in C_Q(q)$, $\ol{u}$ also lies in $C_G(\w)$, for:
$$
\begin{array}{rl}
\ol{u}\,\w\,\ol{u}^{-1}&=\ol{u}\,\ol{q}\,z\,k^{-1}\ol{u}^{-1}\\
	&=\ol{u}\,\ol{q}\,\ol{u}^{-1}\,\up{z}{\ol{u}\,}\,\up{k^{-1}}{\ol{u}\,}\\
	&=\ol{q}\,\delta_q(u)\,\up{z}{\ol{u}\,}\,\up{k^{-1}}{\ol{u}\,}\\
	&=\ol{q}\,k^{-1}\,\up{k}{\ol{u}\,}\,\up{z^{-1}}{\ol{u}\,}\,z\,\up{z}{\ol{u}\,}\,\up{k^{-1}}{\ol{u}\,}\\
	&=\ol{q}\,k^{-1}\,\up{k}{\ol{u}\,}\,\up{z^{-1}}{\ol{u}\,}\,\up{z}{\ol{u}\,}\,\up{k^{-1}}{\ol{u}\,}\,z\\
	&=\ol{q}\,k^{-1}z\\
	&=\ol{q}\,z\,k^{-1}\\
	&=\w
\end{array}
$$
hence $C_G(\w)$ contains  the preimage $\pi^{-1}(C_Q(q))$ of $C_Q(q)$ by $\pi:G\longrightarrow Q$.\hfill$\square$\\

\noindent
{\it Step 2. condition {\rm (i)} and {\rm (ii)} $\implies$ $G$ is icc.}\\
We suppose that $G$ is not icc and prove that either condition (i) or condition (ii) fails. Suppose there exists $u\not=1$ in $G$ with $\up{u}{G}$ finite. If $u\in K$, $u\in FC_G(K)$ and condition {\rm (i)} fails. So suppose that $u=\ol{q}\,k^{-1}$ for some $k\in K$ and $q\not=1$ in $Q$; $q$ necessarily lies in $FC(Q)$. Let $K_0=C_G(u)\cap K$; it has finite index in $K$, and $\forall\,x\in K_0$, $\theta_q(x)=k\,x\,k^{-1}$. If $K_0\not=K$, let $h\in K\setminus K_0$ and $\w=[u,h]\not=1$; then $\w$ lies in $K$ and $\up{\w}{G}$ is finite since $C_G(\w)\supset C_G(u)\cap h\,C_G(u)\,h^{-1}$ has finite index in $G$. Therefore when $K_0\not=K$ condition {\rm (i)} fails. 

Suppose in the following that $K=K_0$, {\it i.e.} $\theta_q$ is inner, $\forall\,x\in K$, $\theta_q(x)=k\,x\,k^{-1}$. Let $Q_0=\proj(C_G(u))$; $Q_0$ is included in $C_Q(q)$. If $C_Q(q)\setminus Q_0\not=\varnothing$, let $p\in C_Q(q)\setminus Q_0$; one has $w=[\ol{p},u]\not=1$ in $G$. Let's prove that $w$ lies in $Z(K)$:
$$
\begin{array}{rl}
[\ol{p},u]&=\ol{p}\,\ol{q}\,k^{-1}\ol{p}^{-1}k\,\ol{q}^{-1}\\
&=\ol{p}\,\ol{q}\,\ol{p}^{-1}\,\up{k^{-1}}{\ol{p}\,}\,k\,\ol{q}^{-1}\\
&=\ol{q}\,\delta_q(p)\,\up{k^{-1}}{\ol{p}\,}\,k\,\ol{q}^{-1}\\
&=k\,\delta_q(p)\,\up{k^{-1}}{\ol{p}\,}\,k\,k^{-1}\\
&=k\,\delta_q(p)\,\up{k^{-1}}{\ol{p}\,}
\end{array}
$$
is conjugated in $K$ to $d_q(p)^{-1}\in Z(K)$; hence $w=d_q(p)^{-1}\in Z(K)\subset K$. Furthermore $w$ has a finite conjugacy class in $G$ since $C_G(w)$ contains $C_G(u)\cap \ol{p}\,C_G(u)\,\ol{p}^{-1}$ which has a finite index in $G$. Hence $w\in FC_G(K)$ and condition {\rm (i)} fails.

Suppose now that $\proj(C_G(u))=C_Q(q)$; let $v\in C_Q(q)$, since $C_G(u)\supset K$  one has $\ol{v}\,u\,\ol{v}^{-1}=u$ in $G$. Then:
$$
\begin{array}{rl}
\ol{q}\,k^{-1}&=\ol{v}\,\ol{q}\,k^{-1}\,\ol{v}^{-1}\\
&=\ol{v}\,\ol{q}\,\ol{v}^{-1}\,\up{k^{-1}}{\ol{v}\,}\\
&=\ol{q}\,\delta_q(v)\,\up{k^{-1}}{\ol{v}\,}\\
\implies \  &\delta_q(v)^{-1} k^{-1}\,\up{k}{\ol{v}\,}=d_q(v)=1~.
\end{array}
$$
Hence for all $v\in C_Q(q)$, $d_q(v)=1$, so that $[q]=0$ in $H^1(C_Q(q),Z(K)))$. Condition {\rm (ii)} fails.
\hfill$\square$

 \subsubsection{Proofs of results of \S \ref{fcqfg2}}\null\hfill\smallskip
\label{preuvehomomorphism}

\noindent
{\it Proof of Proposition-Definition \ref{defhomomorphism}.}
The proof of Proposition-Definition \ref{H1ext} remains valid since for any $q\in FC(Q)$, $C_Q(FC(Q))\subset C_Q(q)$, so that the map $\Xi$ is well defined. It remains to show that $\Xi$ is a homomorphism. Let $q_1,q_2\in \ker\Phi$, such that there exists $k_1,k_2\in K$ with $\forall\,x\in K$, $\theta_{q_1}(x)=k_1\,x\,k_1^{-1}$ and $\theta_{q_2}(x)=k_2\,x\,k_2^{-1}$. Then $q_1q_2\in FC(Q)$ and since $\ol{q_1}\,\ol{q_2}=\ol{q_1q_2}\,f(q_1,q_2)$,
$$\theta_{q_1q_2}(x)=k_1k_2f(q_1,q_2)^{-1}\,x\,f(q_1,q_2)(k_1k_2)^{-1}~.$$
{\it Computation of $\delta_{q_1q_2}$.} Given $u\in Q$,
$$
\begin{array}{rll}
\ol{u}\,\ol{q_1}\,\ol{q_2}\,\ol{u}^{-1}&=\ol{q_1}\,\delta_{q_1}(u)\,\ol{q_2}\,\delta_{q_2}(u)&= \ol{q_1}\,\ol{q_2}\,\up{\delta_{q_1}(u)}{k_2^{-1}}\delta_{q_2}(u)\\
&=\ol{u}\,\ol{q_1q_2}\,f(q_1,q_2)\,\ol{u}^{-1}&=\ol{q_1q_2}\,\delta_{q_1q_2}(u)\,\up{f(q_1,q_2)}{\ol{u}\,}\\
\end{array}
$$
$$
\implies\ \delta_{q_1q_2}(u)= f(q_1,q_2)\,\up{\delta_{q_1}(u)}{k_2^{-1}}\delta_{q_2}(u)\,\up{f(q_1,q_2)^{-1}}{\ol{u}\,}~.
$$
{\it Computation of $d_{q_1q_2}$.} For $u\in Q$, by definition:
$$
\begin{array}{rl}
d_{q_1q_2}(u)&=\up{f(q_1,q_2)}{\ol{u}\,}\,\delta_{q_2}(u)^{-1}\up{\delta_{q_1}(u)^{-1}}{k_2^{-1}} f(q_1,q_2)^{-1}f(q_1,q_2)\,k_2^{-1}k_1^{-1}
\up{k_1}{\ol{u}\,}\,\up{k_2}{\ol{u}\,}\,\up{f(q_1,q_2)^{-1}}{\ol{u}\,}\\
&=\up{f(q_1,q_2)}{\ol{u}\,}\,\delta_{q_2}(u)^{-1}\,\up{\delta_{q_1}(u)^{-1}}{k_2^{-1}}\,k_2^{-1}k_1^{-1}
\,\up{k_1}{\ol{u}\,}\,\up{k_2}{\ol{u}\,}\,\up{f(q_1,q_2)^{-1}}{\ol{u}\,}\\
&=\up{f(q_1,q_2)}{\ol{u}\,}\,\delta_{q_2}(u)^{-1}k_2^{-1}\,\delta_{q_1}(u)^{-1}\,k_1^{-1}
\,\up{k_1}{\ol{u}\,}\,\up{k_2}{\ol{u}\,}\,\up{f(q_1,q_2)^{-1}}{\ol{u}\,}\\
&=\up{f(q_1,q_2)}{\ol{u}\,}\,\delta_{q_2}(u)^{-1}k_2^{-1}\,d_{q_1}(u)\,\up{k_2}{\ol{u}\,}\,\up{f(q_1,q_2)^{-1}}{\ol{u}\,}\\
&=\up{f(q_1,q_2)}{\ol{u}\,}\,d_{q_1}(u)\,\delta_{q_2}(u)^{-1}k_2^{-1}\,\up{k_2}{\ol{u}\,}\,\up{f(q_1,q_2)^{-1}}{\ol{u}\,}\\
&=\up{f(q_1,q_2)}{\ol{u}\,}\,d_{q_1}(u)\,d_{q_2}(u)\,\up{f(q_1,q_2)^{-1}}{\ol{u}\,}\\
&=d_{q_1}(u)\,d_{q_2}(u)
\end{array}
$$
(keep in mind that $\forall\,q\in \ker\Phi,\ \forall\,u\in Q$, $d_{q}(u)\in Z(K)$.)\\
Hence $[q_1\,q_2]=[q_1]\,[q_2]$ which proves that $\Xi$ is a homomorphism. \hfill$\square$\\

\noindent{\it Proof of Proposition \ref{extfcqfg}.}
 If $FC(Q)$ is finitely generated then $C_Q(FC(Q))$ has finite index in $Q$ for it is the intersection of the centralizers of the finite family of generators of $FC(Q)$. Suppose now that  $C_Q(FC(Q))$ has finite index in $Q$.
The proof of Theorem \ref{extension} remains valid in such a case if one changes $C_Q(q)$ into $C_Q(FC(Q))$ by noting that under this hypothesis for any $q\in FC(Q)$, $C_Q(q)$ contains $C_Q(FC(Q))$ as a finite index subgroup. 
\hfill$\square$

\section{\bf Semi-direct products}
In this section we particularize Theorem \ref{extension} in the case the extension splits, {\it i.e.} when there exists a section $s:Q\longrightarrow G$ which is a homomorphism; once $K$ and $Q$ are identified with their isomorphic images in $G$, the group $G$ stands for the {\it semi-direct product}  $G=K\rtimes_\theta Q$ 
with associated homomorphism $\theta: Q\longrightarrow Aut(K)$.
 With this
notation, for any $k\in K$, $q\in Q$, $q\,k\,q^{-1}=\theta(q)(k)$ in $G$. We shall write in the following $\theta_q$ instead of $\theta(q)$.
 We denote by ${\Phi}
: FC(Q)\longrightarrow Out(K)$ the homomorphism which makes the following diagram commute:
\[\xymatrix{
Q\ \ar^{\theta}[r] &Aut(K) \ar @{->>}[d]\\
\overset{}{FC(Q)}\ar^{\Phi}[r] \ar@{^{(}->}[u] & Out(K)\\
 }
\]
\noindent
We reach a new formulation of the necessary and sufficient condition 
given by Theorem \ref{extension} that rephrases condition (ii) and can also be expressed in terms of infiniteness of orbits of $\theta(Q)$ in some subgroups of $K$ and injectivity of the restriction of $\theta$ to $FC(Q)$.

\subsection{Statement and rephrasing of the result}\label{split_results} 
The main result for semi-direct products with infinite conjugacy classes is:
\begin{thm}[icc semi-direct product]\label{semidirect}
Let $G=K\rtimes_\theta Q\not=\{1\}$; $G$ is icc  if and only if:
\begin{itemize}
\item[(i)] $FC_G(K)=\{1\}$, and
 \item[(ii)] For all $q\in\ker\Phi\setminus\{1\}$ and $k\in K$ with  $\forall\,x\in K$,\,
$\theta_q(x)=k\,x\,k^{-1}$,
 $k$ has an infinite $\theta(Q)$-orbit.
\end{itemize}
\end{thm}
\noindent
The result is proved in the next section.

\begin{remark} \label{split_orbit}
Let $\vartheta : G\longrightarrow Aut(K)$ be  the
homomorphism that extends $\theta$ to $G$, {\it i.e.} $\forall\, g\in G,\, k\in K$, $\vartheta(g)(k)=g\,k\,g^{-1}$, and $\vartheta_K:K\longrightarrow Inn(K)$ its restriction to $K$; the diagram below commutes.
\[
\xymatrix{
K \ar@{->>}^{\vartheta_{K}}[r]\ar@{_{(}->}[rd] & Inn(K)
\ar@{_{(}->}[rd]
\\
& G \ar@{->>}[r]^{\vartheta} &\vartheta(G)  \ar@{^{(}->}[r]& Aut(K)\\
Q \ar@{^{(}->}[ru] \ar@{->>}[r]_{\theta} &\theta(Q) \ar@{^{(}->}[ru]
}
\]

\noindent
Define:
$$
K_\theta:=\vartheta_K^{-1}(Inn(K)\cap\theta(FC(Q)))
$$
$K_\theta$ is a subgroup of $K$ that is preserved under the action of $\theta(Q)$. Then Theorem \ref{semidirect} rephrases as:\smallskip\\
{\bf Rephrasing of Theorem \ref{semidirect}.}\ {\sl
Let $G=K\rtimes_\theta Q\not=\{1\}$;
Then $G$ is icc if and only if the following conditions hold:
\begin{itemize}
\item[(i)] $\theta(Q)$ has only infinite orbits in $FC(K)\setminus\{1\}$,
 \item[(ii.a)] $\theta(Q)$ has only infinite orbits  in  $K_\theta\setminus\{1\}$,
\item[(ii.b)]   $\theta:Q\longrightarrow Aut(K)$ restricted to $FC(Q)$ is injective,
\end{itemize}
}
\end{remark}

\begin{example}
In particular, whenever  $K\setminus \{1\}$ contains no finite $\theta(Q)$-orbits, $G$ is icc if and only if $\theta:FC(Q)\longrightarrow Aut(K)$ is injective.\smallskip\\
\noindent For example
The semi-direct product $K\rtimes Aut(K)$ is icc if and only if $Aut(K)$ has no finite orbits in $K\setminus \{1\}$; {\it e.g.} the group of rigid  motions  $\R^n\rtimes O(n)$ of the $n^{th}$-dimensional euclidian space is icc, while none of its discrete subgroups are icc, since by the Bieberbach theorem, they are all virtually Abelian (see Proposition \ref{finite_index}).\medskip
\end{example}

\subsection{Proof of Theorem \ref{semidirect}}  We prove the equivalence here between condition (ii) of Theorem \ref{semidirect} and condition (ii) of Theorem \ref{extension}. We proceed in two steps.\smallskip\\
{\it Step 1. Conditions (i) and (ii) of Theorem \ref{extension} $\implies$  condition $(ii)$ of Theorem \ref{semidirect}.}\\ Suppose that condition $(ii)$ of Theorem \ref{semidirect} does not hold: there exists $q\in FC(Q)\setminus\{1\}$ such that $\theta_q(x)=kxk^{-1}$  $\forall\,x\in K$ for some $k$ with finite $\theta_q$-orbit. Then $Q_0=\theta^{-1}(Stab_{\theta(Q)}(k))$ has a finite index in $Q$, and let $Z_0=Q_0\cap Z_Q(q)$; $Z_0$ has a finite index in $Q$.\smallskip\\
\noindent{\sl First case.} Suppose $Z_Q(q)\setminus Z_0\not=\varnothing$ and let $u\in Z_Q(q)\setminus Z_0$; since $[u,q]=1$, necessarily $[u,k]\in Z(K)$, thereby since $u\not\in Z_0$, there exists $z\not=1$ in $Z(K)$ such that $uku^{-1}=kz$. Then $Z_G(z)\supset K$, and since $z=k^{-1}uku^{-1}$, $Z_G(z)\supset Q_0\cap u\,Q_0\,u^{-1}$. Hence $Z_G(z)$ has a finite index in $G$ therefore $FC_G(K)\not=\{1\}$: condition (i) fails.\smallskip\\
\noindent{\sl Second case.}
Suppose that $Z_0=Z_Q(q)$. For any $u\in Z_Q(q)$, $\theta(u)\in Stab_{\theta(Q)}(k)$ so that $d_q(u)=[k^{-1},u]=1$ and $[q]=0$ in $H^1(Z_Q(q),Z(K))$ (see remark \ref{H1split}); condition (ii) of Theorem \ref{extension} fails.\medskip\\
{\it Step 2. Condition (ii) of Theorem \ref{semidirect} $\implies$ condition (ii) of Theorem \ref{extension}.}\\
Suppose that condition (ii) of Theorem \ref{extension} does not hold: let $q\in FC(Q)\setminus\{1\}$ and $k\in K$ such that $\forall\,x\in K$, $\theta_q(x)=k\,x\,k^{-1}$, and suppose that $[q]=0$ in $H^1(Z_Q(q),Z(K))$. There exists $z\in Z(K)$ such that $\forall\,u\in Z_Q(q)$, $k^{-1}u\,k\,u^{-1}=z\,\up{z^{-1}}{u}$ and it follows that $[u,kz]=1$:
$$
\begin{array}{rl}
&k^{-1}u\,k\,u^{-1}=z\,\up{z^{-1}}{u}\\
 \iff &z^{-1} k^{-1}u\,k\,u^{-1}\,\up{z}{u}=1\\
 \iff &(k\,z)^{-1}u\,k\,z\,\,u^{-1}=1
\end{array}
$$
 Let $k'=kz$, then $\theta_q(x)=k'\,x\,k'^{-1}$
 and since $Z_Q(q)$ has a finite index in $Q$, $k'$ has a finite $\theta(Q)$-orbit. Condition $(ii)$ of Theorem \ref{semidirect} fails.\hfill$\square$\\

\section{\bf Wreath products}

Throughout the section, $D$, $Q$ are groups and $\Omega$ is a $Q$-set, {\it i.e.} a set equipped with a left $Q$-action.  Let $G$ be the {\it complete} (or {\it unrestricted}) {\it wreath product} denoted by $G=D\wr_\Omega Q$, {\it i.e.}
let $D^{\Omega}$ denote the group of maps from $\Omega$ to
$D$  and let $\lambda : Q\longrightarrow
Aut(D^{\Omega})$ be the homomorphism  defined by $\forall x\in \Omega$,
$\forall\f\in D^{\O}$, $\lambda(q)(\f)(x)=\f(q^{-1}x)$ ;  the
group $G$ is the split extension $G=D^{\O}\rtimes
Q$ associated with $\lambda$, in the sense that $\forall \f\in
D^{\O}$, $\forall q\in Q$, $q\f q^{-1}=\lambda(q)(\f)$. When $\Omega=Q$ with $Q$ acting by  multiplication on the left
 one talks of the {\it complete regular wreath product} denoted by $D\wr Q$.

We also consider the {\it restricted wreath product} $G=D\wr_{\O r} Q$: let $D^{(\O)}$ be the group of maps from $\Omega$ to $D$ with finite support, and define as above $G$ as the split extension $G=D^{(\O)}\rtimes Q$ associated with $\lambda$; $G$ is a subgroup of $D\wr_\Omega Q$ which is countable whenever $D$ and $Q$ are countable. When $\Omega=Q$, one talks of the {\it restricted regular wreath product} $D\wr_r Q$.

\subsection{Statement of the results}

We obtain different results for restricted and complete wreath products. The two following results concern restricted wreath products.

\begin{thm}[icc restricted wreath products]\label{wreath_prod}
Let  $G=D\wr_{\Omega  r} Q$, with $D\not=\{1\}$ ; a necessary and
sufficient condition for $G$  to be icc is that,  on the one hand
at least one of the
following conditions holds:
\begin{itemize}
 \item[(i.a)] $D$ is icc,
 \item[(i.b)] all
$Q$-orbits in $\O$ are infinite.
\end{itemize}
and on the other hand the following condition holds:
\begin{itemize}
\item[(ii)]  1 is the only element of $FC(Q)$ which fixes $\O$
pointwise, \end{itemize}
\end{thm}

\noindent
\begin{cor}[restricted regular wreath products icc]
When the $Q$-action on $\O$ is free (in particular, for $D\wr_r Q$), $G$ is
icc if and only if either $D$ is icc or $Q$ is infinite.
\end{cor}

We now turn to the similar results for complete wreath products.

\begin{thm}[icc complete wreath products]\label{wreath_prod2}
Let  $G=D\wr_{\Omega  r} Q$, with $D\not=\{1\}$ ; a necessary and
sufficient condition for $G$  to be icc is that,   on the one hand
at least one of the
following conditions holds:
\begin{itemize}
 \item[(i.a)] $D$ is icc,
 \item[(i.b)] all
$Q$-orbits in $\O$ are infinite.
\end{itemize}
and on the other hand the two following conditions hold:
\begin{itemize}
\item[(ii)]  1 is the only element of $FC(Q)$ which fixes $\O$
pointwise,
\item[(iii)] $D$ is centerless.
 \end{itemize}
\end{thm}

\begin{cor}[complete regular wreath products icc]
When the $Q$-action on $\O$ is free (in particular, for $D\wr Q$), $G$ is
icc if and only if either $D$ is icc or $Q$ is infinite and $D$ is centerless.
\end{cor}

\begin{example} The lamplighter group $\ZnZ{2}\wr_r\Z$ is icc while the complete wreath product $\ZnZ{2}\wr \Z$ is not. Let $\Omega=\ZnZ{n}$ be equipped with the natural $\Z$-action $(p,q\mod n)\mapsto p+q\mod n$, then $\ZnZ{n}\wr_\Omega\Z$ and $\ZnZ{n}\wr_{\Omega,r}\Z$ are not icc  (conditions $(i.a)$ and $(i.b)$ fail).
\end{example}


\begin{cor} The icc property is stable under wreath product: any complete (respectively restricted) wreath product of icc groups is icc.
\end{cor}

\subsection{Proofs of the results}
In both cases $G$ is seen as a semi-direct product and the property icc is discussed using the reformulation of Theorem \ref{semidirect} 
in remark \ref{split_orbit}. 
\subsubsection{Proof of Theorem \ref{wreath_prod}} Let $K=D^{(\O)}$; 
$G=K\rtimes_\theta Q$. First observe that $FC(K)=FC(D)^{(\O )}$, hence $FC(K)=\{1\}$ if and only if $FC(D)=\{1\}$. Condition (i) in remark \ref{split_orbit} is equivalent to $FC_G(K)=\{1\}$. Let's prove that here:
$$
FC_G(K)=\{1\}\quad\iff\quad \left\lbrace\begin{array}{l}FC(D)=\{1\}\\
\text{or}\\
\text{all $Q$-orbits in $\O$ are infinite}
\end{array}\right.
$$
Indeed, suppose that $FC_G(K)=\{1\}$; if $FC(D)\not=\{1\}$ let $u\not=1$ lying in $FC(D)$. If $\O$ would contain a finite $Q$-orbit $Q.\w$ then the map from $\O$ to $D$ which equals $u$ in $\w$ and $1$ everywhere else would lie in $FC_G(K)\setminus\{1\}$; impossible. This proves the necessary part of the assumption. Reciprocally: if $FC(D)=\{1\}$ 
then $FC(K)=\{1\}$ and therefore $FC_G(K)=\{1\}$. If $FC(D)\not=\{1\}$ and $\O$ contains only infinite $Q$-orbits, let $\eta\in FC(D)^{(\O)}\setminus\{1\}$, then $Q.{\rm supp}(\eta)$ is infinite but also equals the finite union of finite sets $\cup_{q\in Q} {\rm supp}(\null^q\eta)$; impossible. This proves the sufficient part of the assumption. 
Hence condition (i) in remark \ref{split_orbit} is equivalent to conditions (i.a) or (i.b) of Theorem \ref{wreath_prod}.

Secondly, $K_\theta=Z(K)=Z(D)^{(\O)}$ since here all automorphisms in $\theta(Q)$ that are inner are therefore the identity.  Here condition (ii.a) in remark \ref{split_orbit} is a particular case of its condition (i).

Finally, $\theta :Q\longrightarrow Aut(K)$ is injective if and only if $1$ is the only element of $Q$ that fixes $\O$ pointwise. Condition (ii.b) in remark \ref{split_orbit} is equivalent to condition (ii) of Theorem \ref{wreath_prod}. We conclude with the formulation of Theorem \ref{semidirect} in remark \ref{split_orbit}.\hfill$\square$

\subsubsection{Proof of Theorem \ref{wreath_prod2}}
Let $K=D^{\O}$; 
$G=K\rtimes_\theta Q$ and keep in mind the proof of Theorem \ref{wreath_prod}. Here again $FC(K)=FC(D)^{\O }$. Condition (i) in remark \ref{split_orbit} is equivalent to $FC_G(K)=\{1\}$. Let's prove that here:
$$
FC_G(K)=\{1\}\quad\iff\quad \left\lbrace\begin{array}{l}FC(D)=\{1\}\\
\text{or}\\
\text{all $Q$-orbits in $\O$ are infinite and $Z(K)=\{1\}$}
\end{array}\right.
$$
As above $FC_G(K)=\{1\}$ implies that either $FC(D)=\{1\}$ or $\O$ contains only infinite $Q$-orbits. But moreover necessarily $Z(D)=\{1\}$ for otherwise let $z\not=1$ lying in $Z(D)$ the map from $\O$ to $D$ constant equal to $z$ would lie in $Z(K)$ and thereby in $FC_G(K)$. This proves the necessary part. Reciprocally; if $FC(D)=\{1\}$ then $FC_G(K)=\{1\}$  so suppose in the following that $FC(D)\not=\{1\}$. Suppose that $\O$ contains only infinite $Q$-orbit and that $FC_G(K)\not=\{1\}$; let $\eta:\O\longrightarrow FC(K)$ be an element of $FC_G(K)\setminus\{1\}$. Since $\null^Q\eta$ is finite and $Q$ has no finite orbits in $\O$, necessarily ${\rm supp}(\eta)$ is infinite. Moreover its conjugacy class $\null^K\eta$ contains all $\eta':\O\longrightarrow FC(K)$ such that $\forall\,x\in\O$, $\eta'(x)$ and $\eta(x)$ are conjugate in $D$. In particular $\null^K\eta$ is infinite once for an infinite set of elements $x$ in $\O$, $\null^D\eta(x)$ is not a singleton. Hence necessarily $Z(D)\not=\{1\}$. This proves the sufficient part. The remaining of the proof goes the same way as for Theorem \ref{wreath_prod}.\hfill$\square$

\section{\bf Examples}\label{pc}

We now look at some particular cases and look how Theorem \ref{extension} rephrases. Among them extensions where the factors verify some additional hypothesis, but also groups containing a finite index subgroups, amalgamated products and HNN extensions. In those last two cases we recover briefly the main results in \cite{ydc} that answer two questions of Pierre de la Harpe (cf. \cite{pdlh}).


\subsection{Finite extensions}
\label{particulardebut}
We consider here extensions with finite quotients.
\begin{prop}[icc finite extension]\label{qfinite}
Let $G$ be a finite extension:
$$
1\longrightarrow K \longrightarrow G\longrightarrow Q\, \text{finite}\longrightarrow 1
$$
Then $G$ is icc if and only if $K$ is icc and $\Theta:Q\longrightarrow Out(K)$ is injective.
\end{prop} 

\begin{proof} On the one hand, if $G$ is icc, then necessarily $K$ is icc; for suppose on the contrary that $\exists\,k\in K\setminus\{1\}$ such that $C_K(k)$ has finite index in $K$; since $K$ has finite index in $G$ and $C_K(k)\subset C_G(k)$, $C_G(k)$ has finite index in $G$,  $\up{k}{G}$ is finite and $G$ is not icc. On the other hand, suppose that $K$ is icc; since $Q$ is finite, $FC(Q)=Q$ so that (see 2\textsuperscript{nd} item in example \ref{examplesextension}),  $G$ is icc if and only if $\Theta:Q\longrightarrow Out(K)$ is injective.
\end{proof}

\begin{example}
It follows from Proposition \ref{qfinite} and from the 4\textsuperscript{th} item in example \ref{examplesextension} that: {\sl if $Q$ is finite simple},\\
\indent -- {\sl $G$ is icc if and only if $K$ is icc and the extension is not equivalent to $K\times Q$.}\smallskip\\
For example let $p$ be a prime integer; a group containing an icc normal subgroup $K$ with index $p$ is either icc or isomorphic to $K\times \ZnZ{p}$.
\end{example}

\subsection{Finite index subgroups}
\label{finiteindex}

Let $H$ be a finite index subgroup of $G$ and:
 $$\ol{H}:=\bigcap_{g\in G}
g\,H\,g^{-1}$$ 
then $\ol{H}$ is  the maximal subgroup of $H$ normal in $G$ and has finite index in $G$; let $Q=G/\ol{H}$. 
Denote by $\vartheta : G\longrightarrow Aut(\ol{H})$ the homomorphism defined by $\forall g\in G, \forall\,k\in \ol{H}$,
$\vartheta(g)(k)=\up{k}{g}$.\medskip

\begin{prop}[finite index subgroup and icc]\label{finite_index} Let $G$ be a group, $H$ a finite index subgroup of $G$ and $\vartheta:G\longrightarrow Aut(\ol{H})$ as above; it induces $\Theta:Q\longrightarrow Out(\ol{H})$.
 Then:\\
\phantom{$\iff$}  $G$ is icc\\
$\iff$ $H$ is icc and $\Theta:Q\longrightarrow Out(\ol{H})$ is injective,\\
$\iff$  $H$ is icc and $\forall\ g\in G\setminus H$ with a finite order,
 $\vartheta(g)$ is not the identity\\
$\iff$  $H$ is icc and $\forall\ g\in G\setminus H$ with a finite order, $\vartheta(g)$ is not inner.
\end{prop}

\begin{proof}
First note that every finite index subgroup  in an icc group is icc; for suppose that $H$ is a finite index subgroup of $G$ and
that $H$ is not icc: there exists $h\in H\setminus\{1\}$ with $\up{h}{H}$ finite, hence $C_H(h)$ has  finite index in $H$, and since $H$ has  finite index in $G$ and $C_G(h)\supset C_H(h)$, $C_G(h)$ has a finite index in $G$ and $G$ is not icc.

 Now suppose that $H$ is icc, so that $\ol{H}$ is icc; applying proposition \ref{qfinite} to the extension of $\ol{H}$ by $Q$ one obtains the first assumption. In particular, if $G$ is icc, then $\forall\,g\in G\setminus H$, $\vartheta(g)$ is not inner, therefore in particular $\vartheta(g)\not=Id$.

Now suppose that $G$ is not icc, necessarily there exist $g\in G\setminus \ol{H}$ and $k\in \ol{H}$ such that $\forall\,x\in \ol{H},\,\pi(g)(x)=\up{x}{k}$. Let $\w=g\,k^{-1}$,  $\w\in G\setminus H$ and $\vartheta(\w)$ is the identity. Since $\ol{H}$ has a finite index in $G$, there exists $n>1$ such that $\w^n\in\ol{H}$. Since $\ol{H}$ is icc, necessarily $\w^n=1$.
But let $g\in G\setminus H$ such that $g^n=1$ and $\exists\, k\in\ol{H}$ with $\forall\,x\in\ol{H}$, $\vartheta(g)(x)=\up{x}{k}$ and let $u=gk^{-1}$. Then $\vartheta(u)$ is the identity on $\ol{H}$ and $u^n=g^nk^{-n}=k^{-n}$; necessarily $u^n=1$ for $u^n$ has a finite conjugacy class lying  in $\ol{H}$. This proves the last two assumptions.
\end{proof}

\begin{example}
Virtually nilpotent groups are not icc, since nilpotent groups have non-trivial center (Theorem 5.34, \cite{rotman}). By a celebrated theorem of Gromov, finitely generated groups with polynomial growth are not icc.
\end{example}

\begin{cor}
If $G\setminus H$ contains no torsion element (in particular, when $G$ is torsion-free), then $G$ is icc if and only if
$H$ is icc.
\end{cor}

\subsection{Extensions with an Abelian factor}
\label{particularfin}\label{section:qabelien}
We consider here the cases of extensions where either the kernel or the quotient  is an Abelian group.
\begin{prop}[icc extension of Abelian group]\label{kernelAbelian}
Let $G\not=\{1\}$ be an extension:
$$
1\longrightarrow K \text{Abelian} \longrightarrow G\longrightarrow Q\longrightarrow 1
$$
and $\theta:Q\longrightarrow Aut(K)$ be the associated homomorphism.

Then $G$ is icc if and only if both:
\begin{itemize}
\item[(i)] $FC_G(K)=\{1\}$,
\item[(ii)] the restricted homomorphism $\theta:FC(Q)\longrightarrow Aut(K)$ is injective.
\end{itemize}
\end{prop}

\begin{proof}
Since $Inn(K)=\{1\}$ the coupling $\Theta:Q\longrightarrow Out(K)$ defines a homomorphism $\theta:Q\longrightarrow Aut(K)$. Given $u\in K$, $\up{u}{G}$ coincides with the $\theta(Q)$-orbit of $u$. Clearly if either (i) or (ii) fails then $G$ is not icc. Conversely suppose $G$ is not icc, and let $u\not=1$ with $\up{u}{G}$ finite. If $u\in K$ then (i) fails. If $u\in G\setminus K$ then let $q=\proj(u)$; necessarily $q\in FC(Q)\setminus\{1\}$. Let $K_0=K\cap C_G(u)$, it has  finite index in $K$  and $\theta(q)$ restricts to the identity on $K_0$. If $K_0\not=K$, let $v\in K\setminus K_0$ and $w=[v,u]$. Then $w\in K$, $w\not=1$ and $\up{w}{G}$ is finite since its centralizer contains $C_G(u)\cap v\,C_G(u)\,v^{-1}$   which has  finite index in $G$. Hence either (i) fails or $\theta(q)$ is the identity and consequently (ii) does not hold.
\end{proof}

\begin{example}
Consider a {\it metabelian} group $G\not=\{1\}$, {\it i.e.} whose derived group $[G,G]$ is Abelian, and $G_{\rm ab}=G/[G,G]$. Then $G$ is icc if and only if $\theta:G_{\rm ab}\longrightarrow Aut([G,G])$ is injective and $[G,G]\setminus \{1\}$ contains only infinite $\theta(G_{\rm ab})$-orbits, if and only if $[G,G]\setminus \{1\}$ contains only infinite $\theta(G_{\rm ab})$-orbits and $[G,G]$ is a maximal Abelian subgroup (see also example below).\medskip
\end{example}

When the quotient is Abelian, Theorem \ref{extension} becomes:

\begin{prop}[icc extension by Abelian]\label{qAbelian}
Let $G\not=\{1\}$ be a group which decomposes as an extension by an Abelian group with  associated coupling $\Theta:Q\longrightarrow Out(K)$:
$$1\longrightarrow K\longrightarrow G\overset{\pi}{\longrightarrow} Q\ \text{Abelian}\longrightarrow 1$$
{\rm 1)} Then $G$ is icc if and only if:
\begin{itemize}
\item[(i)] $FC_G(K)=\{1\}$, and
\item[(ii)] $\Xi:\ker\Theta\longrightarrow H^1(Q,Z(K))$ is injective.
\end{itemize}
if and only if {\rm (i)} holds and:
\begin{itemize}
\item[(ii$'$)] $G$ is centerless
\end{itemize}
{\rm 2)} In case  $Q$ is moreover infinite cyclic, then $G$ is icc if and only if {\rm (i)} 
holds and:
\begin{itemize}
\item[(ii$''$)] $\Theta: Q\longrightarrow Out(K)$ is injective.
\end{itemize}
\end{prop}

\begin{proof}
Since $Q$ is Abelian one has $FC(Q)=Q$ and $C_Q(FC(Q))=Q$. In particular Proposition \ref{extfcqfg} applies and one obtains  that conditions (i) and (ii) are necessary and sufficient for $G$ to be icc. 

If condition (ii$'$) does not hold clearly $G$ is not icc. Now suppose that conditions (i) holds while $G$ is not icc; it suffices to prove that condition (ii$'$) fails. With Theorem \ref{extension} there exists $q\in \ker\Theta\setminus\{1\}$ such that $[q]=0$ in $H^1(C_Q(q),Z(K))$, and with Lemma \ref{lastlemma}, there exists $\w\in G\setminus K$ with $C_G(\w)=\pi^{-1}(Q)=G$, therefore $Z(G)$ is non-empty. This proves that conditions (i) and (ii$'$) are necessary and sufficient for $G$ to be icc. 

 In case  $Q$ is infinite cyclic it is sufficient to show that the non-injectivity of $\Theta$ implies that $G$ is not icc. 
 Suppose $\Theta$ is non-injective; let $t$ be  a generator of $\Z$, then there exists $n\in\Z$ and $k\in K$ such that $\forall\,x\in K$, $\ol{t}^nx\,\ol{t}^{-n}=k\,x\,k^{-1}$. Hence $C_G(k^{-1}\ol{t}^n)$ contains $K$ and $\ol{t}^n$, and therefore, has a finite index in $G$; $G$ is not icc. 
\end{proof}

\begin{example}\label{nonperfect}
Let $G$ be a {\it non-perfect} group, {\it i.e.} whose derived group $[G,G]$ is a proper subgroup (in particular, when $G$ is {\it solvable}); let $G_{\rm ab}=G/[G,G]$. If $G$ is not Abelian, it decomposes as an extension with kernel $[G,G]$; let $\Theta:G_{\rm ab}\longrightarrow Out([G,G])$ be the associated coupling. Then $G$ is icc if and only if condition (i) of Theorem \ref{extension} holds and $\Xi:\ker\Theta\longrightarrow H^1(G_{\rm ab},Z([G,G]))$ is injective, if and only if condition (i) holds and $G$ is centerless.
If $G$ is Abelian obviously $G$ is not icc.
\end{example}

\subsection{In case the kernel is hyperbolic}
\noindent
For definition and basic facts upon hyperbolic groups, we refer the reader to \cite{delzant}.

\begin{prop}[icc hyperbolic group]\label{hyperbolic}
Let $G$ be a hyperbolic group; then $G$ is icc if and only if $G$ is non-elementary and does not contain a non-trivial finite characteristic (respectively normal) subgroup.
\end{prop}

\begin{proof}
If $G$ is elementary, {\it i.e.} either finite or virtually $\Z$, or if $G$ contains a non-trivial finite normal (in particular characteristic) subgroup,  then clearly $G$ is not icc. Conversely suppose that $G$ is non-elementary and not icc. Since in hyperbolic groups infinite order elements have virtually cyclic centralizers (cf. Corollary 7.2, \cite{delzant}), $FC(G)$ is periodic. Since hyperbolic groups contain finitely many conjugacy classes of torsion elements (cf. Lemma 3.5 in \cite{delzant}), $FC(G)$ is finite and the conclusion holds.
\end{proof}

\begin{prop}[icc extension of hyperbolic group]
Let $G$ be a group extension:
$$1\longrightarrow K \text{hyperbolic}\longrightarrow G\longrightarrow Q\longrightarrow 1$$
with $K\not=\{1\}$. Then $G$ is icc if and only if  both:
\begin{itemize}
\item[(i)] $K$ is icc
\item[(ii)] $\Phi : FC(Q)\longrightarrow Out(K)$ is injective.
\end{itemize}
\end{prop}

\begin{proof} According to Theorem \ref{extension} when conditions (i) and (ii) hold $G$ is icc. Conversely suppose that $G$ is icc. Necessarily, $K$ is non-elementary for otherwise $K$ would be finite or would contain a characteristic infinite cyclic subgroup, which with  Theorem \ref{extension}.(i) and Proposition \ref{FCGN}  would contradict that $G$ is icc. Now a non-elementary hyperbolic group has a finite center (follows from Corollary 7.2, \cite{delzant}). Necessarily, $Z(K)=\{1\}$ for otherwise as above  $G$ would not be icc. Therefore (see first item in example
\ref{examplesextension}) $\Phi$ is injective --condition (ii) holds-- and condition (i) of Theorem \ref{extension} holds. By applying proposition \ref{hyperbolic}, $K$  not icc would imply condition (i) of Theorem \ref{extension}; hence $K$ is icc: condition (i) also holds.
\end{proof}

\subsection{HNN extensions.}   Let $A$ be a non-trivial group with subgroups $C,C'$ and let $\varphi : C\longrightarrow C'$ be an isomorphism. Let $G:=A_{*\varphi}$ be the HNN extension (\cite{ls}): $$G\simeq\,<A,t\,|\, \forall\,c\in C,\, t\,c\,t^{-1}=\varphi(c)>~.$$ It is said to be {\sl degenerate} when $A=C=C'$ and {\sl non-degenerate} otherwise.\\
 Let $\widetilde{C}$ be the largest subgroup of $C\cap C'$ normal in $G$:
\begin{gather*}
C_0=\bigcap_{a\in A}aCa^{-1}\ \cap \ \bigcap_{a\in A}aC'a^{-1};\quad  \widehat{C}_{k+1}=C_k\cap \phi(C_k)\cap \phi^{-1}(C_k)\quad
C_{k+1}=\bigcap_{a\in A}a\widehat{C}_{k+1}a^{-1}\\ \widetilde{C}=\bigcap_{k\in\N}C_k
\end{gather*}

\noindent
Consider the epimorphism $\proj:G\longrightarrow\Z$ such that $A\subset\ker \proj$ and $\proj(t)$ generates $\Z$, and let $K:=\ker \proj$ so that $G$ decomposes as a split extension:
$$
1\longrightarrow K\longrightarrow G\overset{\proj}{\longrightarrow}\Z\longrightarrow 1
$$
and let $\theta:\Z\longrightarrow Aut(K)$ be the associated homomorphism: $\forall\,x\in K$, $\theta(x)=t\,x\,t^{-1}$; it induces the coupling $\Theta:\Z\longrightarrow Out(K)$. When the HNN extension is degenerate, $K=A=\widetilde{C}$.

\begin{prop}[icc HNN extensions]
Let $G=A_{*\varphi}$ be an HNN extension.\smallskip\\
$\bullet$ If the HNN extension is non-degenerate. Then $G$ is  icc if and only if:\smallskip\\
\indent {\rm (i)} $FC_G(\widetilde{C})=\{1\}$. 
\smallskip\\
$\bullet$ If the HNN extension is degenerate.  Then $G$ is icc if and only both {\rm (i)} holds and\smallskip\\
\indent{\rm (ii)} the homomorphism  $\Theta:\Z\longrightarrow Out(A)$ is injective.
\end{prop}

\begin{proof}
If the HNN extension is degenerate: $C=C'=A$ then $\ker \proj=A$ and $G$ is an extension of $A$ by $\Z$. In such case, Proposition \ref{qAbelian}(2) applies: $G$ is icc if and only if Theorem \ref{extension}{\rm (i)} holds and $\Theta$ is injective. The conclusion holds.\smallskip\\ 
If the HNN extension is non-degenerate. We suppose in the following  that, without loss of generality, $C$ is  a proper subgroup of $A$.\smallskip\\
Applications of the Britton's Lemma (\cite{ls}) allow to prove the two following facts:  \smallskip\\
\noindent a) {\it The homomorphism $\Theta$ is injective.} Suppose on the contrary that there exists $n\in\N^*$ and $k\in K$ such that $\forall\,x\in K$, $t^nx\,t^{-n}=k\,x\,k^{-1}$. Let $\a\in A\setminus C$; $t^n\a\,t^{-n}$ is reduced so that $k\not\in A$, $k$ has reduced form:
$$k=k_0t^{\e_1}k_1\cdots t^{\e_p}k_p\quad\text{with}\ k_0,k_i\in A, \e_i=\pm 1,\ \forall\,i=1,\ldots,p
$$
for some $p\geq n$. On the one hand, since $t^nk\,t^{-n}=k$, necessarily $\e_1=-1$ or $\e_p=1$; on the other hand since $t^{-n}k\,t^{n}=k$, necessarily $\e_1=1$ or $\e_p=-1$; it follows that $\e_1=\e_p=\pm 1$. If $\e_1=\e_p=1$ then $k\,\a\,k^{-1}$ is also reduced and $t^n\,\a\, t^{-n}=k\,\a\,k^{-1}$ implies $t^n\in kA$ which cannot occur since $t^n\not\in K$. If $\e_1=\e_p=-1$: similarly $t^{-n}\a\,t^n=k^{-1}\a\,k$; if $\a\in C'$ the right term is reduced while the left term is not which contradicts that $p\geq n$; if $\a\not\in C'$ the same argument as before leads to a contradiction.
\medskip\\
b) {\it $\theta(\Z)$ has only infinite orbits in $K\setminus C\cap C'$.}  Suppose on the contrary that there exists $k\in K\setminus C\cap C'$ with a finite $\theta(\Z)$-orbit. Then $k\in K\setminus A$, for suppose without loss of generality that $k\in A\setminus C$, then $\{t^n\,k\,t^{-n}\,;\,n\in\N\}$ is infinite and contained in the $\theta(Q)$-orbit of $k$. Hence $k$ has reduced form:
$$k=k_0t^{\e_1}k_1\cdots t^{\e_p}k_p\quad\text{with}\ k_0,k_i\in A, \e_i=\pm 1,\ \forall\,i=1,\ldots,p
$$
for some $p\geq 1$.  
 Necessarily $\e_1=\e_p$ for otherwise the $\theta(\Z)$-orbit of $k$  would contain the infinite set $\{t^{\e_1\,n}\,k\,t^{\e_p\,n}\,;n\in\N\}$, so suppose without loss of generality that $\e_1=\e_p=-1$. An immediate induction shows that indeed $\forall\, i=1,\ldots,p$, $\e_i=-1$. This leads to a contradiction since $k\in K$ implies that $\sum_{i=1}^p\e_i=0$.\medskip\\
By applying  Proposition \ref{qAbelian}(2) with a) one obtains that $G$ is icc if and only if $FC_G(K)=\{1\}$. But with b) $FC_G(K)\subset C\cap C'$, and since $\widetilde{C}$ is the largest subgroup of $C\cap C'$ normal in $G$, $FC_G(K)\subset \widetilde{C}$. Hence condition $FC_G(K)=\{1\}$ becomes here $FC_G(\widetilde{C})=\{1\}$; this proves the result.
\end{proof}

\begin{example}
Consider the Baumslag-Solitar group $BS(m,n)=<a,t\,|\,ta^mt^{-1}=a^n>$. Then $G$ is icc if and only if $m\not=\pm n$.
\end{example}

\subsection{Amalgamated product} In this section let $A,B$ be groups, $C,C'$ proper subgroups respectively of $A$ and $B$, $\vf:C\longrightarrow C'$ an isomorphism and consider the amalgamated product (cf. \cite{ls}) $G=A*_{C}B=\, <A,B\,|\,\forall\,c\in C, c=\vf(c)>$. It is said to be {\sl degenerate} when $C,C'$ have index 2 respectively in $A$ and $B$, and {\sl non-degenerate} otherwise. In the following we identify $C$ with a subgroup both of $A$ and $B$.\\
\noindent
Denote by:
$$\widetilde{C}_A=\bigcap_{a\in A} a\,C\,a^{-1}, \quad \widetilde{C}_B=\bigcap_{b\in B} b\,C\,b^{-1},\quad \widetilde{C}=\widetilde{C}_A\cap \widetilde{C}_B$$
 $\widetilde{C}$ is the largest subgroups of $C$ normal in $G$.
 
 \begin{prop}[icc amalgams]
Let $G=A*_{C} B$ be a non-trivial amalgamated product.\smallskip\\
$\bullet$ In case the amalgam is non-degenerate: $G$ is icc if and only if:\\
\indent {\rm (i)}\ $FC_G(\widetilde{C})=\{1\}$.\\
$\bullet$ In case the amalgam is degenerate: here $\widetilde{C}=C$ and $G$ decomposes as an extension:
$$
1\longrightarrow C\longrightarrow G\longrightarrow \Z/2\Z * \Z/2\Z\longrightarrow 1
$$
  and $G$ is icc if and only if both {\rm (i)} holds  and:\smallskip\\
  \indent {\rm (ii)} the associated coupling $\Theta:\ZnZ{2}*\ZnZ{2}\longrightarrow Out(C)$ is injective.
\end{prop}

\begin{proof}
Applications of the normal form theorem  (cf. \cite{ls}) allow to prove the following facts:  \\
\noindent
a) {\it Every element of $A\setminus \widetilde{C}$ has an infinite $G$-conjugacy class.} Let $a\in A\setminus C$ and $b\in B\setminus C$. Then the elements $(ab)^n\,a\,(ab)^{-n}$, $n\in \N$ are pairwise disjoint elements of $\up{a}{G}$, so that $\up{a}{G}$ is infinite. Now let $x\in A\setminus \widetilde{C}$, a conjugate of $x$ lies in $(A\cup B)\setminus C$ so that  $\up{x}{G}$ is infinite.\medskip\\
b) {\it If $[A:C]>2$ then every element of $G\setminus (A\cup B)$ has an infinite conjugacy class.} Let $u\in G\setminus (A\cup B)$; up to conjugacy by an element of $A\cup B$ we suppose that $u$ has reduced form $u=a_1b_1\cdots a_nb_n$ for some $n\in\N^*$ with $a_1\in A$, $b_1\in B\setminus C$ and $\forall\,i=2,\ldots ,n$: $a_i\in A\setminus C$ and $b_i\in B\setminus C$. Let $b\in B\setminus C$ and $a\in A\setminus C$ such that $a$ and $a_1^{-1}$ lie in different cosets of $A/C$. Then the elements $(ba)^n\,u\,(ba)^{-n}$ for $n\in\N$ are pairwise distinct elements in $\up{u}{G}$, so that $\up{u}{G}$ is infinite.\smallskip\\
We now distinguish two cases:\smallskip\\
{\it First case: the degenerate case, $[A:C]=[B:C]=2$.} In such a case, $C=\widetilde{C}$ is normal in $A,B$ and $G$. In case $C=\{1\}$, $G=\ZnZ{2}*\ZnZ{2}$ is not icc and $\Theta:\ZnZ{2}*\ZnZ{2}\longrightarrow Out(C)$ is non-injective; otherwise $G$ splits as:
$$
1\longrightarrow C\longrightarrow A*_CB\longrightarrow \ZnZ{2}*\ZnZ{2}\longrightarrow 1~,
$$
$FC(\ZnZ{2}*\ZnZ{2})$ consists in the characteristic infinite cyclic subgroup. Note that $\Theta:\ZnZ{2}*\ZnZ{2}\longrightarrow Out(C)$ is injective if and only if it is once restricted to $FC(\ZnZ{2}*\ZnZ{2})$. Note also that whenever $\Theta$ is non-injective,  $G$ is not icc, since any  element $\not=1$ of $FC(\ZnZ{2}*\ZnZ{2})$ generates a finite index subgroup of $\ZnZ{2}*\ZnZ{2}$. Together with  a) and Theorem \ref{extension}, one obtains that $G$ is icc if and only conditions {(i)} and (ii) hold.\smallskip\\
{\it Second case: the non-degenerate case.} If $\widetilde{C}=\{1\}$ then on the one hand $FC_G(\widetilde{C})=\{1\}$ and on the other hand the facts a) and b) above prove that $G$ is icc.\\ If $\widetilde{C}\not=\{1\}$, denote $\ol{A}=A/\widetilde{C}$, $\ol{B}=B/\widetilde{C}$ and $\ol{C}=C/\widetilde{C}$; $G$ decomposes as an extension:
$$
1\longrightarrow \widetilde{C}\longrightarrow A*_CB\longrightarrow \ol{A}*_{\ol{C}}\ol{B}\longrightarrow 1
$$
and with the Correspondence Theorem (Theorem 2.28, \cite{rotman}), on the one hand, $[\ol{A}:\ol{C}]=[A:C]$ and $[\ol{B}:\ol{C}]=[B:C]$ and on the other $\ol{C}$ does not contain any non-trivial subgroup normal in $\ol{A}*_{\ol{C}}\ol{B}$. With the facts a) and b) proved above,  $\ol{A}*_{\ol{C}}\ol{B}$ is icc. Therefore by applying  Theorem \ref{extension} together with a), $G$ is icc if and only if $FC(\widetilde{C})=\{1\}$. 
\end{proof}

\begin{example}
 A free product of non-trivial groups is either icc or an infinite dihedral group $\ZnZ{2}*\ZnZ{2}$. \\
\end{example}

\begin{center}{\textbf{Acknowledgments}}
\end{center}
The author wishes to thank Pierre de la Harpe for having introduced  him to the problem, for corrections on a preliminary draft version, 
for frequent conversations on the subject and for all his useful remarks and comments.\\ \\

\end{document}